\documentclass{amsart}
\usepackage[utf8]{inputenc}
\usepackage{amsmath,amssymb,enumitem,tikz,mathrsfs,amsthm,dirtytalk}
\usepackage[english]{babel}
\usepackage{hyperref}

\newtheorem{theorem}{Theorem}
\newtheorem{definition}{Definition}
\newtheorem{lemma}{Lemma}
\newtheorem{prop}{Proposition}
\newtheorem{corollary}{Corollary}

\title{Ostrowski's ``On Dirichlet Series and Algebraic Differential Equations''}
\author{Erik Christian Hansen, Yonathan Stone, and Jesse Wolfson}
\thanks{This translation was supported in part by NSF Grant No. DMS-1944862}
\address{Department of Mathematics, University of California-Irvine}
\email{echanse1@uci.edu}
\email{ystone@uci.edu}
\email{wolfson@uci.edu}

\begin{document}

\maketitle

\maketitle

\vspace{16pt}

\begin{center}
		Original Bibliographic Information
\end{center}
\begin{itemize}
	\item[] \textbf{Author:} Alexander Ostrowski
	\item[] \textbf{Title:} \"Uber Dirichletsche Reihen und algebraisehe Differentialgleichungen.
	\item[] \textbf{Year:} 1920
	\item[] \textbf{Language:} German
	\item[] \textbf{Journal:} Math. Zeit., vol. 8, pp. 241-298
\end{itemize}
\vspace{16pt}

\tableofcontents
	
\newpage
\section*{Introduction}

After Hölder's proof that the \textit{gamma function} does not satisfy any algebraic differential equation
\footnote{Math. Ann., 28 (1887), pp. 1-13, for other proofs of the theorem see E.H. Moore, Math. Ann., 48: 49-74 (1897); N. Nielsen, Handbook of the theory of the gamma function, Leipzig, Teubner (1906); A. Ostrowski, Math. Ann. 79 (1919), pp. 286-284.},
the same was shown for many other functions. Among these, Hilbert's theorem that the \textit{Riemann $\zeta$-function} does not satisfy any algebraic differential equation
\footnote{Hilbert, Mathematical Problems, C. R. du 2. congrès international des math., Paris 1902, p. 100.
The execution of the proof goes to V.E.E. Stadigh, Dissertation, Helsingfors 1902.
For references in the literature on functions that do not satisfy any algebraic differential equation, reference should be made to the introduction to Stadigh's dissertation and the treatise by R.D. Carmichael.} belongs to some of the most interesting results.
In the first section of this treatise I will prove that this same property holds for all Dirichlet series
$$\sum\frac{a_n}{n^s}$$
\textit{whose coefficients $a_n$ are all non-zero}, or
that \textit{the indices $n$ of all non-zero coefficients have infinitely many prime divisors}. More generally, this property is shared by all Dirichlet series

$$\sum a_n e^{-\lambda_ns}$$
possessing the property that none of the $\lambda_n$ can be expressed as integer linear combinations of finitely many of the other $\lambda_n$, that is to say that infinitely many of the $\lambda_n$ are linearly independent over the integers.
We will proceed to prove this theorem in an even more generalized form (\hyperlink{Theorem 1}{Theorems 1} and \hyperlink{Theorem 6}{6}), as we will prove the non-existence of an algebraic \textit{difference differential equation}, that is, a functional equation of the form
$$F(x, f(x), f'(x),\ldots,f^{(\nu)}(x), f(x+h_1),\ldots,$$$$ f^{(\nu_1)}(x+h_1),\ldots,f(x+h_\mu),\ldots,f^{(\nu_\mu)}(x+h_\mu)) = 0$$
where $h_1,h_2,\ldots,h_n$ are real numbers, and F is a polynomial of the arguments listed above.
This is of importance insofar that we can show that the Riemann $\zeta$-function does not satisfy any algebraic \textit{difference} differential equations.  This in turn allows us to verify the open claim conjectured by Hilbert roughly 20 years ago, namely that the function $\zeta(x,s)$, defined by the series:
$$\zeta(x,s)=\frac{x}{1^s}+\frac{x^2}{2^s}+\frac{x^3}{3^s}+\ldots$$
for $|x|\leq1,\text{Re}(s)>1$ fails to satisfy any algebraic partial differential equation when viewed as a two-variable function in x and s
\footnote{Mathematical problems, 1. c. P. 101.}.
The proof of this fact is given in \S2 (\hyperlink{Theorem 2}{Theorem 2}) using the
functional equation
$$x\frac{\partial\zeta(x,s)}{\partial x} = \zeta(x,s-1)$$
provided by Hilbert in 1. c. for which $\zeta(x,s)$ is a solution. We then provide a second proof of this using
neither Hilbert's functional equation nor the results obtained in the first section. In doing so we end up proving the non-existence of algebraic PDEs admitting solutions of the form
$$\sum \frac{f_n(x)}{n^s}$$
where the $f_n(x)$ are arbitrary polynomials of degree n and whose convergence conditions are fulfilled (\hyperlink{Theorem 3}{Theorem 3}).
In \S3 we draw on the properties verified in \S2 to establish further consequences.
For instance, it follows that it is impossible to construct $\zeta(x,s)$ using iterated composition of arbitrary analytic functions of one variable and arbitrary algebraic functions of several variables (\hyperlink{Theorem 4}{Theorem 4})
\footnote{This question was motivated by Hilbert's well-known theorem regarding the existence of analytic functions of three variables which cannot be constructed via iterated composition of two-variable functions. 1. c. P. 92}.
The eliminations required for this proof can be carried out in an entirely strict and particularly simple and clear fashion using a (of course long-established) algebraic lemma, for which I have included a proof based on field-theoretical observations in \S4, and
which may be used in a similar manner in many related pursuits. After the investigations in \S1 have been somewhat supplemented in \S4, we will more generally examine the properties of analytic functions which are also solutions to algebraic differential equations, i.e. \textit{``algebraic-transcendental functions''}, in \S5 .
The results in \S5 then allow us to apply the theorems from \S1 and \S4 to power series.  This results in, for example, the theorem that one can obtain power series not satisfying any algebraic differential equations in infinitely many ways by simply multiplying certain terms by $-1$ (\hyperlink{Theorem 11}{Theorem 11}).  In addition we obtain a few more results, of which I would like to highlight the theorem that the series
$$\frac{x}{1^s}+\frac{x^2}{2^s}+\frac{x^3}{3^s}+\ldots$$
fails to satisfy any algebraic differential equation for any constant rational values of $s$.
The same property for \textit{irrational} values of $s$ is likely verifiable without much additional work, however I failed produce any rigorous proof of this statement.
In the next section (\S6), however, I am able to deduce that the set of values of $s$ for which $\zeta(x,s)$ does satisfy an algebraic differential equation is \textit{countable} as a consequence of a more general proposition. In \S6 I tackle the more general question of which Dirichlet series satisfy algebraic differential equations in the first place. Using the results from the first two sections, it follows that such Dirichlet
series of the type
$$\sum \frac{a_n}{n^s}$$
and must be of the form
$$F(n_1^{-s},n_2^{-s},\ldots,n_\mu^{-s}),$$
where $F(x_1,x_2,\ldots,x_\mu)$ is a power series in $\mu$ arguments. Moreover, one can easily construct such power series $F$ for which each choice of the integers $n_1,n_2,\ldots,n_\mu$ 
represent functions in $s$ which satisfy algebraic differential equations.
This follows since the existence of functions $F(x_1,x_2,\ldots,x_\mu)$ with the property that $F(\varphi_1(x),\varphi_2(x),\ldots,\varphi_\mu(x))$ satisfies an algebraic differential equation can be established as soon as $\varphi_1(x),\varphi_2(x),\ldots,\varphi_\mu(x)$ satisfy the same equation.
We then prove that a function $F(x_1,x_2,\ldots,x_\mu)$ possesses this last property if and only if it satisfies a so-called Mayer system\footnote{Translator's note: the author frequently uses the German word \textbf{System} when talking to refer to objects one could refer to as \textbf{collections} in modern English.  We have elected to translate it quite literally as ``system'' to highlight the use of ostensibly now outdated language in mathematical writing.} of algebraic partial differential equations (\hyperlink{Theorem 15}{Theorem 15}). 
(Our method of proof results in an even stronger result.)
If a general Dirichlet series of the type 
$$\sum a_n e^{-\lambda_n s}$$
satisfies an algebraic differential equation, we can use the results from \S1 and \S2 to represent it in the form 
$F(e^{-\lambda_1s},e^{-\lambda_2s},\ldots,e^{-\lambda_\mu s})$,
where $F$ is a power series iterated over both positive and negative powers of its arguments and $\lambda_1,\lambda_2,\ldots,\lambda_\mu$ are positive linearly independent real numbers.
For Dirichlet series of the type 
$$\sum a_ne^{-\lambda_ns},$$
if the power series $F$ is iterated over
positive powers of its arguments, and only in this case, one can establish the existence of an analytic function which is represented by the power series $F$ in a given domain of convergence. 
In any case, even given the existence of such a function $F$, the fact that a linearly independent system of values $\lambda_1,\lambda_2,\ldots,\lambda_\mu,$
gives us a Dirichlet series satisfying an algebraic differential equation cannot be used to conclude that it possesses the stated property in the theorem above.
However, it can be shown that it satisfies an algebraic partial differential equation (\hyperlink{Theorem 19}{Theorem 19}), and moreover this property is still preserved even if the power series $F(x_1,x_2,\ldots,x_\mu)$ does not have $2\mu$-dimensional domain of convergence.
In this case, there exists an algebraic partial differential equation \textit{formally} satisfied by the power series $F(x_1,x_2,\ldots,x_\mu)$.
On the other hand, it can be proven that if the function $F(e^{-\lambda_1s},e^{-\lambda_2s},...,e^{-\lambda_{\mu}s})$ satisfies an algebraic differential equation for each linearly independent system of values consisting of the positive $\lambda$, then the same holds for
$$F(\varphi_1(s),\varphi_2(s),\ldots,\varphi_\mu(s))$$
as soon as the $\varphi_1(s),\ldots,\varphi_\mu(s)$ satisfy the same algebraic differential equation (\hyperlink{Theorem 20}{Theorem 20}).  I have provided a proof of this statement at the conclusion of \S7.
These last results can also be adapted to the integrals of algebraic differential equations, whose expansions contain arbitrary powers of $x$, which are ideas I also delve into in \S7. 
These results are closely related to the methods for investigating the integrals of
systems of ordinary differential equations of the first order near critical points due to Poincaré, Picard and others\footnote{E. Picard, Treaty of Analysis, 2 (Ed. 2), S. 1-40}.\\

In \S8 I will finally prove (\hyperlink{Theorem 21}{Theorems 21} and \hyperlink{Theorem 22}{22}) that no power series
$$\sum a_i x^{n_i}, (n_{i-1}<n_i<\ldots\text{ ad inf. })$$
in which the $n_i$ are linearly independent over the integers, satisfies an \textit{analytic} differential equation near $x=0$. 
(For complex $n_i$ a few more prerequisites are necessary.)
Using this theorem we give an example of such severe (branching) singularities, that the functions associated with them cannot even be obtained as solutions of \textit{analytic} differential equations at the aforementioned points. 
The proof proceeds in exactly the same way as in \S1, with the
exception of a step that is entirely obvious in the case of algebraic differential equations. This results from us dealing with a phenomenon, which in particular includes the theorem that any analytic function satisfying an analytic differential equation also satisfies an analytic differential equation in which it is not a \textit{singular integral.} 
A direct proof of this fact presents difficulties, since Weierstrass's preparation theorem can famously only be applied to power series in several variables after a suitable linear transformation.  However, such transformations are not applicable for differential equations.
Therefore, we must take a detour to circumvent this difficulty.\\

The requisite material we invoke in our proofs is very elementary.
From the theory of Dirichlet series we use only the uniqueness theorem, according to which a convergent Dirichlet series $\sum a_i e^{-\lambda_i s}$, whose terms are given in ascending order of the $\lambda$, with the property that as a function of $s$ it vanishes for sufficiently large positive $s$, also has the property that all the coefficients $a_i$ vanish.
Otherwise, most of the proofs can be reduced to simple formal results from algebra\footnote{The present treatise is a copy of an inaugural dissertation accepted by the Philosophical Faculty of the University of Göttingen.}.
\newpage

\section{On a class of Dirichlet series with the property of satisfying no algebraic difference differential equations.}
\begin{definition}
An algebraic difference differential equation is a functional equation of the form
\begin{align}
   F(x,f^\nu(x+h_\mu))=0
\end{align}

where $h_1,h_2,\ldots$ are finitely many real numbers and $F$ is a polynomial in its arguments.
\end{definition}

\hypertarget{Proposition 1}{}
\begin{prop}
If an analytic function $\varphi(x)$ satisfies an algebraic difference differential equation of the form (1), it also satisfies an algebraic difference differential equation
$$F(f^{(v)}(x+h_\mu))=0$$
in which $F$ does not explcitly depend on the independent variable $x$.
\end{prop}
\begin{proof}
We may assume that F is an irreducible degree $n$ polynomial in its arguments.
Differentiating (1) in $x$, we obtain a new algebraic difference differential equation for $\varphi$,
$$F^*(x,f^{(\nu)}(x+h_\mu))=0$$
where $F^*(x,f^{(\nu)}(x+h_\mu))$ is a polynomial in its arguments \textit{whose degree is no greater than $n$.}
On the other hand, $F^*$ contains at least one expression of the form $f^{\nu}(x+h_\mu)$ due to differentiation which does not occur in $F$. 
Therefore $F^*$ cannot be divisible by $F$, and since $F$ is irreducible, the resultant of the polynomial $F$ and $F^*$ must be a polynomial $F^{**}(f^{(\nu)}(x+h_\mu))$ that does not vanish identically. 
If not, the polynomials $F$ and $F^*$ would have a common divisor which is entire with respect to $x$ and rational with respect to $f^{\nu}(x+h_\mu)$.
Therefore, according to the famous theorems concerning the unique factorization of polynomials into irreducible factors, they would also have a common factor which is entire in $x$ and  $f^{\nu}(x+h_\mu)$, which is impossible.
Thusly, we consider the identity, 
$$F^{**} = M_1F + M_2F^{*}$$
where $M_1,M_2$ are polynomials in $x, f^{\nu}(x+h_\mu)$.
 Therefore, $\varphi(x)$ also satisfies the algebraic difference differential equation
 $$F^{**}(f^{(\nu)}(x+h_\mu))=0$$
 in which x does not show up explicitly as an argument, as desired.
\end{proof}

Via iterated applications of the same procedure used above, one can prove without any additional machinery that \textit{if an analytic function $ \varphi (x, y, z, \ldots) $ satisfies an algebraic partial differential equation, it also satisfies an algebraic partial differential equation not explicitly depending on $ x, y , z, \ldots $.}

\hypertarget{Proposition 2}{}
\begin{prop}
Let $\tau,\kappa$ be positive integers and $k_1,k_2,\ldots,k_\kappa$ distinct real numbers. Now consider the function $$L(\lambda) = \sum c_{t,i}\lambda^te^{\lambda k_i}\hspace{30pt}(t=0,1,\ldots,\tau; i = 1,2,\ldots,\kappa)$$ in which not all $c$ vanish. Then this expression only has finitely many real roots as a function of $\lambda$.
\end{prop}
\begin{proof}
Assuming $L(\lambda)$ has an infinite number of real roots $\lambda_1,\lambda_2,\ldots$, we can always select a subsequence of those roots converging to $+\infty$ or $-\infty$. If not, the $\lambda_i$ have an accumulation point in the finite reals, and since $L(\lambda)$ is holomorphic, it follows that $L(\lambda)$ vanishes identically in $\lambda$. Thus we may assume the numbers $1,2,3,\ldots$ as our $\lambda_1,\lambda_2,\lambda_3,\ldots$. 
We already know that the sequence
$$ \lambda_1, \lambda_2, \lambda_3,...  \hspace{30pt} (\lambda_i \neq 0, i = 1,2,3,...)$$
possesses the property that
$\lim_{i\to\infty}\lambda_i=\infty$.
(The case $\lim_{i\to\infty}\lambda_i=-\infty$ is reducible to this one by considering $L(-\lambda)$ instead of $L(\lambda)$).
Let $c_{t_1,i_1}\lambda^te^{\lambda k_{i_1}}$ be the term from $L(\lambda)$ with $c_{t_1,i_1}\neq0$ having the largest value of $k_i$, and for this $k_i$ also has the largest value of $t$. By considering the equation
$$\frac1{c_{t_1,i_1}}\lambda_i^{-t_1}e^{-\lambda_ik_{i_1}}L(\lambda_i)=1+\Sigma'\frac{c_{t,p}}{c_{t_1,i_1}}\lambda^{t-t_1}_ie^{\lambda_i(k_p-k_{i_1})}=0$$
and taking the limit as $i \to \infty$, we reach a contradiction.
\end{proof}

Now let $\varphi(s) = \sum a_i e^{-\lambda_i s}$ be a Dirichlet series, where we require no assumptions regarding convergence, however we will assume $\lim\limits_{i \to \infty} \lambda_i = +\infty$.
Let $F(f^{(v)}(s+h_\mu))=0$ be an algebraic difference differential equation not explicitly depending on $s$.
Now replace $f(s)$ with $\varphi(s)$ in this equation, calculate the corresponding resultant formally, and rearrange it into a Dirichlet series.
If all the terms of the resulting Dirichlet series vanish, we say that $\varphi(s)$ \textit{formally} satisfies the equation $F(f^{(v)}(s+h_\mu))=0$.\\

In the sequel $\lambda$, as it appears in expressions of the form $e^{-\lambda s}$, will be referred to as its \textit{exponent}.
\hypertarget{Proposition 3}{}
\begin{prop}
Suppose the Dirichlet series 
$$\varphi(s) = \sum_0^\infty a_ie^{-\lambda_is}\hspace{30pt}(\lambda_0 < \lambda_1 < \ldots \text{ ad inf.)}$$ formally satisfies an algebraic difference differential equation \begin{align}
    F(f^{(\nu)}(s+h_\mu))=0
\end{align} not explicitly depending on $s$.  It follows that for a $\lambda_i$ with sufficiently large index $i$ each subsequent $\lambda_i$ can be expressed as an integer linear combination of the previous $\lambda_i$.\footnote{Translator's note:  This is to say that all but finitely many of the $\lambda$ are linearly dependent over the integers.}
\end{prop}
\begin{proof}
Let $n$ denote the (total) degree of the polynomial $F$ in the arguments $f^{\nu}(s+h_\mu)$. 
We may assume that $n$ is the smallest possible number such that $\varphi(s)$ fails to satisfy any algebraic difference differential equation of total degree less than $n$.
We proceed to construct the following expressions
$$F_{\varrho,\sigma}(f^{(\nu)}(s+h_\mu)=\frac{\partial F(f^{(\nu)}(s+h_\mu))}{\partial f^{(\varrho)}(s+h_\sigma)}.$$
By considering one particular $F_{\varrho,\sigma}$, we observe that it is a polynomial in $f^{(\nu)}(s+h_\mu)$ of degree no greater than  $n$.
Therefore, by substituting the series $\varphi(s)$ for $f(s)$ in $F_{\varrho,\sigma}$ and formally rewriting the result as a Dirichlet series, it must follow that not all terms of the resulting series vanish.
Let  $b_{\varrho,\sigma}e^{-\lambda_{\varrho,\sigma}s}$ denote the first non-vanishing term of the resulting Dirichlet series. 
We will calculate these initial terms for each pair of numbers $\varrho,\sigma$ under consideration.\\

Since the $\lambda_i$ eventually become positive for all indices $i$ past a certain point, and while constructing the $F_{\varrho,\sigma}(\varphi^{(\nu)}(s+h_\mu))$ the individual terms of the series $\varphi^{(\nu)}(s+h_\mu)$ are multiplied with each other at most $(n-1)$ times, it is clear that only finitely many terms of the series $\varphi^{(\nu)}(s+h_\mu)$ can contribute to the construction of the terms 
$b_{\varrho,\sigma}e^{-\lambda_{\varrho,\sigma}s}$ or preceding terms in $F_{\varrho,\sigma}(\varphi^{(\nu)}(s+h_\mu))$, which by assumption vanish. Thus there exists a positive number $\Lambda$ such that if we truncate the series $\varphi(s)$ at any term with the exponent $\lambda_i>\Lambda$ and substitute the resulting expression consisting of only finitely many terms into $F_{\varrho,\sigma}(f^{(\nu)}(s+h_\mu))$ for $f(s)$, the result will again have initial terms $b_{\varrho,\sigma}e^{-\lambda_{\varrho,\sigma}s}$.\\

Denote the smallest of the numbers $\lambda_{\varrho,\sigma}$ by $\Lambda_1$, and let
$\lambda_{\varrho_1,\sigma_1}=\lambda_{\varrho_2,\sigma_2}=\ldots=\lambda_{\varrho_m,\sigma_m}=\Lambda_1$, while for the remaining numbers we assume $\lambda_{\varrho,\sigma}>\Lambda_1$.

Next, for $\lambda_i>\Lambda$, we set 
$$\varphi(s) = A_i(s)+R_i(s),\hspace{30pt}A_i(s) = \sum_0^{i-1}{a_te^{-\lambda_ts}},\hspace{30pt}R_i(s) = \sum_i^\infty{a_te^{-\lambda_ts}}.$$
If we plug the expression $\Phi(s) = A_i(s)+R_i(s)$ into $F(f^{(\nu)}(s+h_\mu))=0$ in the place of $f(s)$ and expand using Taylor's theorem, we obtain
\begin{align}
    F(\varphi^{(\nu)}(s+h_\mu))=F(A_i^{(\nu)}(s+h_\mu))+\sum F_{\varrho,\sigma}(A_i^{(\nu)}(s+h_\mu))\frac{d^{\varrho}R_i(s+h_\sigma)}{ds^\varrho}
 \end{align}
$$+\sum_2\Psi_2(R_i^{(\nu)}(s+h_\mu))F_2(A_i^{(\nu)}(s+h_\mu))+\ldots$$

Here $F_2,F_3,\ldots$ are the second, third, etc. derivatives of $F$ in its arguments, while the $\Psi_2,\Psi_3,\ldots$ are given \textit{homogeneous} forms of the second, third, etc. degree in $R_i^{(\nu)}(s+h_\mu)$. Thus, the exponents of the initial terms of $\Psi_2(R^{(\nu)}(s+h_\mu)),\Psi_3(R^{(\nu)}(s+h_\mu)),\ldots$ are at least $2\lambda_i,3\lambda_i,\ldots$, etc. On the other hand, the total degrees of the polynomials $F_2,F_3,\ldots$ are at most $n-2,n-3,\ldots$, so the lower bounds $2\lambda_i - (n-2)|\lambda_0|,3\lambda_i - (n-3)|\lambda_0|,\ldots$ for the exponents of the initial terms of the third, fourth, etc. expression in (3).

Next, by considering the second expression in the right-hand side of (3), we get the following expression for its initial term \begin{align}
    a_i e^{-(\Lambda_1+\lambda_i)s}
\sum{b_{\varrho_t,\sigma_t}(-\lambda_i)^{\varrho_t}e^{-\lambda_ib_{\sigma_t}}},
\end{align} provided the sum 
$$\sum{b_{\varrho_t,\sigma_t}(-\lambda)^{\varrho_t}e^{-\lambda b_{\sigma_t}}}$$ remains non-zero for $\lambda = \lambda_i$. According to \hyperlink{Proposition 2}{Proposition 2}, however, this sum only has finitely many real roots when regarded as a function of $\lambda$. By labeling its greatest root $\Lambda_2$, it follows that for $\lambda_i>\Lambda+|\Lambda_1|+|\Lambda_2|+n|\lambda_0|$
the term (4) will not cancel with any term of second, third, etc. expression from the right side of (3). 
But since, by assumption, every component on the right side of (3) vanishes, a term with the exponent $\Lambda_1+\lambda_i$ can be found among those in $F(A_i^{(\nu)}(s+h_\mu))$.
Therefore the exponent $\Lambda_1+\lambda_i$ can be written as an integer linear combination of the exponents $\lambda_0,\lambda_1,\ldots,\lambda_{i-1}$. On the other hand, $\Lambda_1$ can also be decomposed as a linear combination of $\lambda_0,\lambda_1,\ldots,\lambda_{i-1}$, given that $\lambda_i > \Lambda_1 + n|\lambda_0|.$ It follows that $\lambda_i$  is an integer-linear function of the $\lambda_0,\lambda_1,\ldots,\lambda_{i-1}$.
\end{proof}
Assuming the conditions of \hyperlink{Proposition 3}{Proposition 3} are fulfilled, all $\lambda_i$ admit representations as integer linear combinations of the exponents $\lambda$ smaller than $\Lambda+|\Lambda_1|+|\Lambda_2|+n|\lambda_0|+1$.
We will thus say that the system of exponents $\lambda_i$ belonging to $\varphi(s)$ \textit{possesses a finite linear basis\footnote{Translator's note: Unless otherwise specified, linear independence is assumed to be over the integers.} from among the numbers $\lambda_i$ themselves.}
Let $\lambda_j$ be the largest $\lambda$ such that $\lambda_j <\Lambda+|\Lambda_1|+|\Lambda_2|+n|\lambda_0|+1$. It bears mentioning that evidently $\lambda_0,\lambda_1,\ldots,\lambda_j$ need not be linearly independent.
Say the number of these $\lambda$ that are linearly independent is $\alpha$. Then it is well-known that we can construct $\alpha$ many integer linear transformations $\omega_1,\omega_2,\ldots,\omega_\alpha$ from $\lambda_0,\lambda_1,\ldots,\lambda_j$ so that, conversely, $\lambda_0,\lambda_1,\ldots,\lambda_j$ and therefore all $\lambda_i$ can be expressed as integer linear combinations of $\omega_1,\omega_2,\ldots,\omega_\alpha$.
Taking these expressions of the $\lambda_i$ and plugging them into the series $\varphi(s)$, we can also rewrite $\varphi(s)$ as
$$\Phi(e^{-\omega_1s},e^{-\omega_2s},\ldots,e^{-\omega_\alpha s}),$$
where $\Phi(x_1,x_2,\ldots,x_u)$ is a power series expanding along both positive and negative integral powers of $x_1,x_2,\ldots,x_u$, although it can only be considered as formal for the time being as we currently still assume nothing regarding the convergence of $\varphi(s)$. - In doing so, we may assume all $\omega_i$ to be positive.\\

From \hyperlink{Proposition 3}{Proposition 3} it follows that there cannot be an infinite number of linearly independent exponents $\lambda_i$ if $\varphi(s)$ formally satisfies equation (2).
To truly appreciate the full scope of this result, we will consider Dirichlet series of type \begin{align}
    \sum \frac{a_n}{n^s}=\sum a_ne^{-s \text{log}(n)}
\end{align}\\

By the uniqueness of the prime factorization of the natural numbers, it follows that the system of logarithms of all prime numbers is linearly independent.
If, among all the numbers $\text{log}(n)$ that actually occur in (5), there are only finitely many that are linearly independent, then, starting from a certain $n$, the numbers $\text{log}(n)$ appearing in (5) are linearly dependent on a certain finite system of logarithms of natural numbers $\text{log}(n)$, say $\text{log}(n_1),\text{log}(n_2),\ldots,\text{log}(n_j)$
But then it follows that all denominators $n$ that actually show up in (5) are only made up of the same prime factors that occur in $n_1,n_2,\ldots,n_j$.
Thus, if the indices $n$ corresponding to non-zero $a_n$ in (5) are divisible by an infinite number of different prime numbers, then (5) surely cannot formally satisfy any equation (2).
However, if the denominators actually occurring in (5) only contain finitely many prime factors $p_1,p_2,\ldots,p_\beta$, we can formally write (5) in the form
$$\Psi(p_1^{-s},p_2^{-s},\ldots,p_\beta^{-s})$$
where $\Psi(x_1,x_2,\ldots,x_\beta)$ is a power series containing only positive powers of $x_1,x_2,\ldots,x_\beta$.
However, if (5) possesses a domain of convergence (hence also a domain of absolute convergence), then $\Psi(x_1,x_2,\ldots,x_\beta)$ also has a domain of absolute convergence in which it represents an analytic function of $x_1,x_2,\ldots,x_\beta$.
\hypertarget{Theorem 1}{}
\begin{theorem}
If a Dirichlet series 
$$\varphi(s)=\sum a_i e^{-\lambda_i s}$$
has a domain $\Omega$ of absolute convergence and if the analytic function it represents in $\Omega$ satisfies an algebraic difference differential equation, then the system of numbers $\lambda_i$ possesses a finite linear basis from among the numbers $\lambda_i$.
\end{theorem}
\begin{proof}
If the analytic function $\varphi(s)$ in $\Omega$ satisfies an algebraic difference differential equation, then, according to \hyperlink{Proposition 1}{Proposition 1}, it also satisfies an equation of the form (2) as in \hyperlink{Proposition 3}{Proposition 3}.
If we substitute $\varphi(s)$ into said equation, and consider that additionally the series $\varphi^{(\nu)}(s+h_\mu)$ converge absolutely in a certain domain and represent the analytic function $\varphi^{(\nu)}(s+h_\mu)$ there, since the formally defined product of finitely many absolutely convergent Dirichlet series also has a domain of absolute convergence and is represented by the product of the corresponding functions on it, we may formally conduct our calculations with the series $\varphi(s)$.
The resulting series $D(s)$ has an absolute domain of convergence and in it represents the result of substituting $\varphi(s)$ into the left hand side of the difference differential equation.
But since this result vanishes identically, we have due to the uniqueness theorem that the series $D(s)$ also vanishes identically.
Therefore  $\varphi(s)$ \textit{formally} satisfies an algebraic difference differential equation of the type as in equation (2) in \hyperlink{Proposition 3}{Proposition 3} and finally the application of Proposition 3 completes the proof of this theorem.
\end{proof}
\newpage

\section{Proof that the series $x+\frac{x^2}{2^s}+\frac{x^3}{3^s}$ and analogously constructed series have the property of not satisfying any algebraic partial differential equation.}
We can now provide the proof of the proposition that Hilbert designated as ``probable''in his Paris lecture on ``Mathematical Problems:''
\hypertarget{Theorem 2}{}
\begin{theorem}
As a function of the variables $x, s$, the function 
$$\zeta(x,s) = x+\frac{x^2}{2^s}+\frac{x^3}{3^s}+\ldots$$
fails to satisfy any algebraic partial differential equation.
\end{theorem}
\begin{proof}
One can immediately verify that the function $\zeta(w,s)$ satisfies the following functional equation provided by Hilbert:
\begin{align}x\frac{\partial \zeta(x,s)}{\partial x}=\zeta(x,s-1).\end{align}
From which one can instantly derive the more general functional equations
\begin{align}(x\frac{\partial}{\partial x})^\mu\zeta(x,s-\nu)=\zeta(x,s-\mu-\nu)\end{align}
and
\begin{align}(x\frac{\partial}{\partial x})^\mu(\frac{\partial}{\partial s})^\lambda\zeta(x,s-\nu)=\frac{\partial^\lambda\zeta}{\partial s^\lambda}(x,s-\mu-\nu).\end{align}
Next, let $\zeta(x,s)$ satisfy an algebraic partial differential equation
$$\Phi(\zeta_{\mu,\lambda}(x,s))=0,$$
where $\zeta_{\mu,\lambda}(x,s)$ denotes $\frac{\partial^\mu}{\partial x^\mu}\frac{\partial^\lambda}{\partial s^\lambda}\zeta_{\mu,\lambda}(x,s)$, and $\Phi$ is a polynomial in its arguments, which by \hyperlink{Proposition 2}{Proposition 2} we can assume does not explicitly depend on $x,s$.\\

The functional equation (6) now allows the derivatives of $\zeta(x,s)$ to iteratively be expressed by the quantities $\zeta(x,s-\nu)$: 
\begin{align}
    \frac{\partial^\mu\zeta(x,s)}{\partial x^\mu}=\frac1{x^\mu}\zeta(x,s-\mu)+c_{\mu,\mu-1}\zeta(x,s-\mu+1)+\ldots+c_{\mu,1}\zeta(x,s-1)
\end{align}
Here, the coefficients $c$ are rational functions of $x$. Clearly, we may solve for the quantities $f(x,s-\nu)$ in the equations (9): \begin{align}
    \zeta(x,s-\mu)=x^\mu\frac{\partial^\mu\zeta(x,s)}{\partial x^\mu}+\gamma_{\mu,\mu-1}\frac{\partial^{\mu-1}\zeta(x,s)}{\partial x^{\mu-1}}+\ldots+\gamma_{\mu,1}\frac{\partial\zeta(x,s)}{\partial x}
\end{align}
and the $\gamma$ are rational (even entire) functions of x. By differentiating equations (9) and (10) $\lambda$ times with respect to $s$: we obtain: \begin{align}
    \frac{\partial^{\mu+\lambda}\zeta(x,s)}{\partial x^\mu\partial s^\lambda}=\frac1{x^\mu}\frac{\partial^\lambda\zeta(x,s-\mu)}{\partial s^\lambda}+c_{\mu,\mu-1}\frac{\partial^\lambda\zeta(x,s-\mu+1)}{\partial s^\lambda}+\ldots
\end{align}
$$+c_{\mu,1}\frac{\partial^\lambda\zeta(x,s-1)}{\partial s^\lambda},$$
\begin{align}
\frac{\partial^\lambda\zeta(x,s-\mu)}{\partial s^\lambda}=x^\mu\frac{\partial^{\mu+\lambda}\zeta(x,s)}{\partial x^\mu \partial s^\lambda}+\gamma_{\mu,\mu-1}\frac{\partial^{\mu+\lambda-1}\zeta(x,s)}{\partial x^{\mu-1}\partial s^\lambda}+\ldots    
\end{align}
$$+\gamma_{\mu,1}\frac{\partial^{\lambda+1}\zeta(x,s)}{\partial x \partial s^\lambda},$$
which follows from the observation that $c$ and $\gamma$ are independent from $s$.
Here, too, the corresponding systems of equations (11) and (12) provide solutions from each other for every value of $\lambda$.
Next, if one substitutes the expressions in (11) into the polynomial $\Phi(\zeta_{\mu,\lambda)}(x,s))$
instead of the partial derivatives $\frac{\partial^{\mu + \lambda} \zeta(x,s)}{\partial x^\mu \partial s^\lambda}$, we create the polynomial
$$\Psi(\frac{\partial^\varrho\zeta(x,s-\nu)}{\partial s^\varrho})$$
in terms of the expressions $\frac{\partial^\varrho\zeta(x,s-\nu)}{\partial s^\varrho}$, whose coefficients are rational functions of $x$.
If we substitute the arguments of the polynomial $\Psi$ for the expresssions in (12), we get $\Psi$ again. 
Therefore, $\Psi$ cannot vanish identically.\\

Next, we multiply the polynomial $\Psi$ by a power of $x$, one sufficiently large such that $\Psi$ is also entire with respect to $x$, and also divide it by the largest possible power of $(x-1)$ in the case where the resulting polynomial is divisible by $(x-1)$.
Then, setting $x=1$, for sufficiently large s, $\zeta(x,s)$ gives rise to the Dirichlet series
$$\zeta(s) = 1+\frac{1}{2^s}+\frac{1}{3^s}+\frac1{4^s}+\ldots,$$
the expressions $\frac{\partial^\varrho\zeta(x,s-\nu)}{\partial s^\varrho}$ turn into the expressions $\zeta^{(\varrho)}(s-\nu)$, the equation
$$\Psi(\frac{\partial^\varrho\zeta(x,s-\nu)}{\partial s^\varrho})=0$$
subsequently gives us an algebraic difference differential equation for $\zeta(s)$.
According to \hyperlink{Theorem 1}{Theorem 1} and the remarks we have appended to the proof of \hyperlink{Proposition 3}{Proposition 3}, $\zeta(s)$ does not satisfy any algebraic difference differential equation. This proves Theorem 2.
\end{proof}
In proving Theorem 2 we made use of the functional equation (6) satisfied by $\zeta(x,s)$.  However, by directly using the same method that led to the proof of Proposition 3, we are capable of obtaining a much more general result.
\\
Let
$$F(x,s)=\sum\varphi_i(x)e^{-\lambda_is}\hspace{30pt}(\lambda_0<\lambda_1<\ldots \text{ ad inf.})$$
be an absolutely and uniformly convergent series in a region of the $x$-plane and a given $s$-half-plane (we assume the same convergence properties for all partial derivatives with respect to $x$), and let $\varphi_i(x)$ be non-identically vanishing polynomials in $x$ whose exact degrees will be denoted by $m_i$.
Let $F(x,s)$ satisfy an algebraic partial differential equation of the form $\Phi(F_{\mu,\nu}(x,s))=0,$
where the left hand side can be assumed to be a polynomial (independent from $x$ or $s$) in the partial derivatives $F_{\mu,\nu}(x,s)=\frac{\partial^{\mu+\nu}F}{\partial x^\mu\partial s^\nu}$.
Denoting the partial derivatives of $\Phi$ with respect to $F_{\mu,\nu}$ by $\Phi_{\mu,\nu}$, we may assume that $F(x,s)$ does not satisfy a partial differential equation of the form
$$\Phi_{\mu,\nu}(F_{\sigma,\varrho})=0.$$
By substituting the series expansion for $F(x,s)$ into $\Phi_{\mu,\nu}(F_{\sigma,\varrho})$ and arranging the result formally like an ordinary Dirichlet series, one obtains that the resulting series expansion 
\begin{align}
    \sum \psi_i(x)e^{-\kappa_is}
\end{align}
does not vanish identically.
This is due to the series expansion for $F(x,s)$ and all of its partial derivatives converging absolutely in a certain region, implying the series (13) also converges absolutely in a certain region and represents the analytic function $\Phi_{\mu,\nu}(F_{\sigma,\varrho})$ in said region.
Let $ \psi_1 e ^ {- \kappa_1 s} $ be the first term in (13) with $ \psi_i $ not vanishing identically.
By truncating the series $F(x,s)$ at a sufficiently large $i$, this first term will remain the same.

For each of the pairs of values $\mu,\nu$ we consider, we will find each of the aforementioned first terms and from among those choose the smallest exponent $\chi_1$ occurring in them.  Designate this exponent using $\Lambda$ and let it occur in $\Phi_{\mu_\tau,\nu_\tau} \:\:(\tau = 1,2,...)$.  We will designate the corresponding polynomials $\psi_1(x)$ using $\psi^{(\tau)}(x)$, while the precise degrees will be denoted using $m^{(\tau)}$.\\
We now set 
$$F(x,s) = A_i(x,s) + R_i(x,s),$$
$$A_i(x,s)=\sum_0^{i-1}\varphi_t(x)e^{-\lambda_ts},\hspace{30pt}R_i(x,s)=\sum_i^\infty\varphi_t(x)e^{-\lambda_ts},$$
and plug $A_i+R_i$ into $\Phi$.  We subsequently obtain: \begin{align}
    \Phi(F_{\sigma,\varrho})=\Phi(\frac{\partial^{\sigma+\varrho}A_i}{\partial x^\sigma \partial s^\varrho})+\sum_{\mu,\nu}\Phi_{\mu,\nu}(\frac{\partial^{\sigma+\varrho}A_i}{\partial x^\sigma \partial s^\varrho})\frac{\partial^{\mu+\nu}R_i}{\partial x^\mu \partial s^\nu}
\end{align}
$$+\sum_2\Phi_2(\frac{\partial^{\sigma+\varrho}A_i}{\partial x^\sigma \partial s^\varrho})
\Psi_2(\frac{\partial^{\mu+\nu}R_i}{\partial x^\mu \partial s^\nu})+\ldots$$
Here $\Phi_2,\Phi_3,\ldots$ are the second, third, $\ldots$ derivatives of $\Phi$ with respect to its arguments and $\Psi_2,\Psi_3,\ldots$ are certain homogeneous forms of the second, third, $\ldots$ degree in the derivatives of $R_i$, respectively.\\

\textit{We now assume that} $\lim_{i\to\infty}m_i=\infty$, and under this assumption select the terms from the series expansions in the second, third, etc. expression in (14) whose exponents are as small as possible. This term, for sufficiently large $i$, satisfies the property 
\begin{align}
    \sum_\tau \psi^{(\tau)}(x)\frac{\partial^{\mu_\tau}\varphi_i(x)}{\partial x^{\mu_\tau}}(-\lambda_i)^{\nu_\tau}e^{-(\lambda_i+\Lambda)s},
\end{align}
provided it does not vanish identically in $x$. 
This is true since on the one hand, for other values of $ \mu, \nu $, the exponent of the initial term of $ \Phi _ {\mu, \nu} (\frac {\partial ^ {\sigma + \varrho} A_i} {\partial x ^ \sigma \partial s ^ \varrho})$ is greater than $ \Lambda $, on the other hand the exponent of the initial term of $ \Psi_2, \Psi_3, \ldots $ is at least $ 2 \lambda_i $ for sufficiently large $i$.  However,
since now we have that $ \Phi (F _ {\sigma, \varrho}) $ vanishes identically, it follows that for sufficiently large $ i $ the expression (15) equals 0, or is canceled out by some term in $ \Phi (\frac {\partial ^ {\sigma + \varrho} A_i} {\partial x ^ \sigma \partial s ^ \varrho})$.
The second scenario will surely \textit{not} occur infinitely many times, provided: 1) the exponents $\lambda_i$ have no finite linear basis, or 2) $\lim_{i\to\infty}\text{sup}\frac{\lambda_i}{\lambda_i-1}=\infty$.
Given either of these assumptions, we have that\begin{align}
    \sum_\tau\psi^{(\tau)}(x)\frac{\partial^{\mu_\tau}\varphi_i(x)}{\partial x^{\mu_\tau}}(-\lambda_i)^{\nu_\tau}
\end{align}
must vanish for infinitely many $i$. 
By taking the coefficient of the highest power of $x$ in (16) and setting it equal to 0, we obtain an equation of the form \begin{align}
    \Theta(m_i,\lambda_i)=0,
\end{align}
where $\Theta$ is a polynomial in $m_i,\lambda_i$ with coefficients that are constant, independent of $i$, and not all vanishing.\\

This last equation, however, cannot hold for infinitely many $i$, provided either $m_i$ grows faster than any power of $\lambda_i$, or $\lambda_i$ grows faster than any power of $m_i$, i.e. if 
$$\text{either}\lim_{i\to\infty}\frac{log(m_i)}{log(\lambda_i)}=\infty\text{ or }\lim_{i\to\infty}\frac{log(\lambda_i)}{log(m_i)}=\infty$$
hold.
In order to obtain a more precise result, we consider the roots $m_i$ of (17), which become infinite with i, and expand them in terms of powers of $\lambda_i$ in a neighborhood of the infinitely distant point.
We thus obtain finitely many series expansions in decreasing fractional powers of $\lambda_i$, which commence with a positive power of $\lambda_i$.
From this it follows that among the accumulation points of $\frac{log(m_i)}{log(\lambda_i)}$ there must be a finite number of non-zero rational numbers. We summarize this result in
\hypertarget{Theorem 3}{}
\begin{theorem}
Let 
$$F(x,s)=\sum\varphi_i(x)e^{-\lambda_is}\hspace{30pt}(\lambda_0<\lambda_1<\ldots \text{ ad inf. })$$
be an absolutely and uniformly convergent series in a given region of the $x$-plane and a given $s$-half-plane (with the same convergence properties holding for all partial derivatives with respect to $x$), and let $\varphi_i(x)$ be polynomials in $x$ with exact degrees $m_i$, with $m_i$ unbounded for increasing $i$.
Suppose there are no non-zero rational numbers occurring among the accumulation points of $\frac{log(m_i)}{log(\lambda_i)}$ (this is certainly the case if $\lim_{i\to\infty}\frac{log(m_i)}{log(\lambda_i)}$ equals 0 or $\infty$).
Provided one of the two following conditions is met:
\begin{enumerate}
    \item The exponents $\lambda_i$ have no finite linear basis,
    \item $\lim_{i\to\infty}\text{sup}\frac{\lambda_i}{\lambda_{i-1}}=\infty,$
\end{enumerate}
Then $F(x,s)$ does not satisfy any algebraic partial differential equations.
\end{theorem}
What is particularly noteworthy about this theorem is that it only makes assumptions about the degrees of the polynomials $\varphi_i(x)$, but not about their coefficients, apart from the assumptions made regarding convergence. It can evidently be tightened in several directions; in particular, as one can easily prove, the requirement of absolute convergence is insignificant.\\

Since $\lim_{n\to\infty}\frac{log(n)}{log(log(n))}=\infty$, \hyperlink{Theorem 2}{Theorem 2} is contained in Theorem 3.

\newpage

\section{On a class of analytic functions of two variables.}
In this section, we will prove that the functions considered in the previous section cannot be obtained from arbitrary analytic functions of one variable and algebraic functions of several variables via iterated composition. 
The eliminations that show up in this proof (and in subsequent observations within this section) can be carried out in entirely strict and particularly simple and clear fashion using the following lemma, which itself is very easily proven by means of field theoretical observations:
\begin{lemma}\label{lemma:sec3}
Let each of the $ n + 1 $ functions 
$$ f_1 (x_1, \ldots, x_n), f_2 (x_1, \ldots, x_n), \ldots, f_ {n + 1} (x_1, \ldots, x_n) $$
satisfy an algebraic equation whose coefficients are polynomials in $ x_1, \ldots, x_n $ with coefficients from an arbitrary field $K$.
Then the functions $f_1,f_2,\ldots,f_{n+1}$ possess an algebraic relation between one another over the field $K$, i.e. there exists a polynomial $\Phi(y_1,y_2,\ldots,y_{n+1})$ with coefficients from $K$, which are not all zero, such that 
$$\Phi(f_1(x_1,\ldots,x_n),f_2(x_1,\ldots,x_n),\ldots,f_{n+1}(x_1,\ldots,x_n))=0$$
identically in $x_i$.\\
\end{lemma}
We will now introduce the concept of an \textit{improper} function of two variables $x,s$, and in particular, for the purposes of this proof, we shall distinguish improper functions according to differing \textit{ranks}.\\

We shall use the term improper function of \textit{first} rank to denote any analytic function depending on only one variable from among $x,s$.  An improper function of \textit{second} rank $\varphi(x,s)$ denotes any algebraic function of several first rank improper functions, such that they can be defined in the same region of $x,s$. A \textit{third} rank improper function will be defined as an analytic function $\psi(\varphi(x,s))$ of a second rank improper function, while assuming that the domain of $\psi$ contains a portion of the range of $\varphi$, itself corresponding to the shared domains of the first rank functions occuring in $\varphi$. In general, an improper function of $ 2n $-th rank is an algebraic function of several improper functions of $ (2n-1)$-th rank, and an improper function of $(2n + 1)$-th rank is an analytic function of an improper function  of $2n$-th rank. In doing so, however, one must always assume that the corresponding domains and ranges of all functions that occur \textit{fit together}.\\

Any analytic function of two variables that is not an improper function of some finite rank is to be called a \textit{proper function of two variables.}
The improper functions of odd order used in the construction of an improper function, are called its \textit{components}.
Such a component consists of an analytic function of one variable with an improper function of even rank for its argument.
We call this analytic function of one variable the \textit{main function} of the component.
For example, the improper function of fourth rank
$$\varphi(\psi(x)+\varphi(s))+\varphi(x)+\varphi(s)$$
has four components
$$\varphi(\psi(x)+\varphi(s)),\psi(x),\varphi(x),\varphi(s)$$
with the main functions $\varphi(z),\psi(z),\varphi(z),\varphi(z)$.\\

If a component has the form $ \varphi (\Phi (x, s)) $, we use the function $ \varphi ^ {(\kappa)} (\Phi (x, s)) $, where $ \varphi ^ {(\kappa)} (z) $ is the $ \kappa $-th derivative of the main function $ \varphi (z) $, to denote \textit{the $ \kappa $-th derivative of the component}. (That is to say that it is not the same as a partial derivative of the component.)\\

These distinctions become important in the construction of the partial derivatives of an improper function. Let $\Phi(x,s)$ be an improper function of $m$-th rank, and let it consist of $\mu$ components.
If we are to construct the $\kappa$-th partial derivative of $\Phi(x,s)$, it must obviously be an algebraic function of the first $\kappa$ derivatives of all the components of $\Phi$, that is an algebraic function of at most $\kappa\mu$ arguments.
Hence, we have that the first $\frac{(\kappa+1)(\kappa+2)}{2}$ partial derivatives of $\Phi$ depend algebraically on at most $\kappa\mu$ arguments (up to the $\kappa$-th order).
And since the number $\frac{(\kappa+1)(\kappa+2)}{2}$ is greater than $\kappa\mu$ for sufficiently large $\kappa$, it follows from \hyperlink{Lemma 1}{Lemma 1} that every improper function of 2 variables satisfies an algebraic partial differential equation. And it follows in particular that

\hypertarget{Theorem 4}{}
\begin{theorem}
Both the function 
$$\zeta(x,s)=\frac{x}{1^s}+\frac{x^2}{2^s}+\frac{x^3}{3^s}+\ldots$$
and more generally any function that admits a series representation satisfying the conditions of \hyperlink{Theorem 3}{Theorem 3} is a proper function of two variables, so it cannot be obtained by concatenating analytic functions of one variable with algebraic functions of several variables.
\end{theorem}

We can furthermore show that the functions referenced in Theorem 4 cannot even satisfy a partial differential equation whose coefficients are improper functions of $ x $ and $ s $. More precisely,

\hypertarget{Theorem 5}{}
\begin{theorem}
Suppose an analytic function $f(x,s)$ of two variables satisfies a partial differential equation whose left-hand side is a polynomial of the unknown function and its derivatives, and whose coefficients are improper functions of two variables. Then it also satisfies a partial differential equation whose left-hand side is a polynomial in the unknown function and its derivatives with numerical coefficients.
\end{theorem}

\begin{proof}
In what follows we understand $ f_{i, k} $ to mean $\frac{\partial ^ {i + k} f (x, s)}{\partial x ^ i \partial s ^ k} $.  Now let \begin{align}
\Phi(f_{i, k}) = 0, \end{align}
be the differential equation of $m$-th order that is satisfied by $f(x,s)$. As stated, $ \Phi $ is a polynomial in $ f $ and its partial derivatives, whose coefficients are improper functions of $ x, s $.
At the same time, let $ \Phi $ be a differential equation with those properties while being of the smallest possible order. Let $\mu$ denote the number of different components occurring in the coefficients of $\Phi$.
Let $ f_{i, m-i} $ be a derivative of the $m$-th order that actually occurs in $ \Phi $, where $ i $ is chosen to be as large as possible. We denote this partial derivative of $\Phi$ by $S$. Since we assumed $ m $ to be as small as possible, $ S $ is non-zero.\\

Let us now construct the $ \kappa + 1 $ equations that result from (18) by taking the \textit{total derivative} of $ \kappa $ times with respect to $ x $ and $ s $,
$$ \frac{d ^ \kappa \Phi}{dx ^ \kappa} = 0, \frac {d^{\kappa-1} \Phi} {dx ^{\kappa-1} ds} = 0, \ldots, \frac {d ^ {\kappa} \Phi} {ds^\kappa} = 0, $$
and note that they have the form
\begin{align} Sf_{i + \kappa, m-i} = \varphi_1(f_ {e, n}), Sf_ {i + \kappa-1, m-i + 1} = \varphi_2(f_{e, n}), \ldots, Sf_ {i, m-i + \kappa} = \varphi _ {\kappa + 1} (f_{e, n}), \end{align}
where the $ \varphi $ are polynomials in the $ f_ {e, n} $ whose orders $ e + n $ are at most $ m + \kappa $.
Of the $ f_ {e, n} $, whose orders are $ m + \kappa $, $ \varphi_{\kappa + 1} $ contains only those with $ e <i $, $ \varphi_ \kappa $ only those with $ e <i + 1, \ldots $, and $ \varphi_1 $ only those with $ e <i + \kappa $.
Therefore by repeated substitution, starting with the equations (19), we arrive at equations in the form 
\begin{align} S^lf_{i +\kappa, m-i} = \psi_1(f_ {e, n}), S^lf_{i + \kappa-1, m-i + 1} = \psi_2 (f_{e , n}),\ldots,
S^lf_{i, m-i + \kappa} = \psi_{\kappa + 1} (f_{e, n}).\end{align}
Here the arguments $ f_{e, n} $ of the functions $ \psi $ are partial derivatives of $f$ up to the $(m + \kappa)$-th order.  However, among the derivatives of $ (m + \kappa) $-th order, the $ f_ {e, n} $ for which $ e \geq i $, $ n \geq m-1 $, are \textit{not} present as arguments of the $\psi$, and equations (20) serve precisely to express these $ f_ {e, n} $ using the remaining partial derivatives of $ m $ -th order and lower.
If we let $ \kappa $ take on all integer values from 1 to $ \kappa '$ and use the equations (20) iteratively for smaller values of $ \kappa $, we finally obtain the rational expressions for all the $ f_ {e, n} $ with $ m < e + n \leq m + \kappa '$, $ e \geq i $, $ n \geq m-i $ in terms of remaining partial derivatives up to the $(m + \kappa')$-th order.
Let us denote the $ \frac {\kappa '(\kappa' + 3)} {2} $ derivatives $ f_ {e, n} $ with $ m <e + n \leq m + \kappa '$, $ e \geq i $, $ n \geq m-i $ as $ p_1, p_2, \ldots $ (the precise order is irrelevant), the remaining partial derivatives of $ f $ up to the $ (m + \kappa') $-th order as $ q_1, q_2, \ldots $, we obtain representations of the form
\begin{align} S ^ lp_t = \chi_t (q_g) \hspace{30pt} (t = 1,2, \ldots, \frac {\kappa '(\kappa' +3)} {2}). \end{align}
Here $ \chi $ and $ S $ are polynomials in the arguments $ q_g $ with coefficients that are algebraic in the  derivatives of various components of the coefficients of $ \Phi $ (up to $ \kappa'$-th order), that is to say expressions of at most $ \mu(\kappa'+ 1) $ variables.
Let us consider the quantities $ q_g $ as indeterminate and use $K$ to denote the field one obtains by adjoining said indeterminates to the field of complex numbers.  Next, let $x_1,x_2,...$ denote the derivatives of all components of the coefficients of $ \Phi$ occurring in $\chi$ and $S$ as well as the components themselves (as before, the order is arbitrary). Provided we assume $ \kappa '> 2 \mu $, \hyperlink{Lemma 1}{Lemma 1} may be invoked, and we see that there is a relation between the $ \frac {\kappa_t} {S} $, which are functions of the $x$, with the coefficients originating from the field $K$.
Therefore we've established a relation between the $ p_t $
\begin{align} \Psi(p_t, q_g)=0\end{align}
whose left-hand side is a polynomial in the $ p_t $ with coefficients that are non-identically vanishing rational functions of the $q_g$. We can now clearly assume these coefficients to be polynomials in the $ q_g $ and use (22) to obtain an algebraic partial differential equation for $f(x, s)$. 
\end{proof}
It seems reasonable to conjecture that \hyperlink{Theorem 4}{Theorems 4} and \hyperlink{Theorem 5}{5} will still hold if we broaden the definition of improper functions in such a way that the functions of one variable used to construct them need only be continuous, but not analytic, provided the improper functions themselves are analytic; that is to say the functions \hyperlink{Theorem 4}{Theorem 4} concerns itself with cannot be constructed using arbitrary continuous functions of one variable and algebraic functions of several variables. However, I have not succeeded at producing a rigorous proof of this without gaps.\\

\newpage

\section{The main theorem regarding arbitrary convergent Dirichlet series and some extensions of the main theorem.}
\hypertarget{Theorem 6}{}
\begin{theorem}
If a function 
$$ \varphi (s) = \sum a_i e ^ {- \lambda_i s}, $$
represented by a convergent Dirichlet series, satisfies an algebraic difference differential equation, then the system of exponents $ \lambda_i $ possesses a finite linear basis.
\end{theorem}
\begin{proof}
If the analytic function $ \varphi(s) $ satisfies an algebraic difference differential equation, then according to \hyperlink{Proposition 1}{Proposition 1} it also satisfies one of the form
\begin{align}
F(f ^ {(\nu)}(s + h_\mu)) = 0,
\end{align}
where $ F $ is a polynomial not explicitly containing $s$ as an argument. In order to prove our theorem, it therefore suffices to show that the series $ \varphi(s) $ also \textit{formally} satisfies equation (23), since then the assertion from \hyperlink{Proposition 3}{Proposition 3} follows. Let us now assume that the Dirichlet series
\begin{align}
F (f ^ {(\nu)} (s + h_\mu)),
\end{align}
which results from the formal calculations does not vanish identically, and let $ ae ^ {-\lambda s} $ be its first non-vanishing term. It is apparent that only finitely many exponents $ \lambda_i $ of $ \varphi(s) $ contribute to the formation of the exponent $ \lambda $ in (24), as well as the exponents of the preceding terms (24) which have vanished due to canceling. For instance, all exponents $ \lambda_i> n | \lambda_0 | + | \lambda | $, where $ n $ is the total degree of $ F $, certainly do not influence this. If we modify certain terms of $ \varphi (s) $ for which $ \lambda_i> n |\lambda_0| + | \lambda | $, then the initial term of the modified series resulting from these insertation will again be $ ae ^ {- \lambda s} $. If we truncate the series $ \varphi (s) $ at any sufficiently large $ \lambda_i $, the difference between $ \varphi (s) $ and the resulting segment
$$ A_i (s) = \sum_0 ^ {i-1} a_ie ^ {- \lambda_p s} $$
of $ \varphi (s) $ is a function $ R_i (s) $, which due to the uniform convergence of Dirichlet series can be written in the following form:
$$ e ^ {- \lambda_i s} S (s), $$
where $ \lim_ {s \to \infty}S(s) = a_i $. 
 - Here and in the sequel, the designation $\lim_{s \to \infty} $ is to be understood in such a way that the limit for $s$ is taken along the positive reals. - And hence one can write
\begin{align}
\varphi ^ {(\nu)} (s + h_ \mu) = A_i ^ {(\nu)} (s + h_ \mu) + e ^ {- \lambda_i s} S ^ {(\nu, \mu )} (s),
\end{align}
where $\lim_ {s \to \infty} S ^ {(\nu, \mu)} (s)$ is a finite constant. If we now plug both $ \varphi (s) = A_i (s) + e ^ {- \lambda_i s} S (s) $ and the expressions (25) into equation (23), and expand using Taylor's theorem, an equation of the following form is created:
\begin{align}
F (A_i ^ {(\nu)} (s + h_ \mu)) + e ^ {- (\lambda_i-n | \lambda_0 |) s} \Phi (s) = 0,
\end{align}
where $ \lim_ {s \to \infty} \Phi (s) $ is finite. On the other hand, $ F (A_i ^ {(\nu)} (s + h_\mu)) $ can be written in the form
$$ e ^ {- \lambda s} P(s), $$
where $ \lim_ {s \to \infty} P (s) = a \neq 0 $, provided $ \lambda_i> | \lambda | + n | \lambda_0 | $. But if we plug this into (26), multiply by $ e ^ {\lambda s} $ and then let $ s $ converge to $ +\infty $, we find that $a=0$ as soon as $ \lambda_i> | \lambda | + n | \lambda_0 | $, which contradicts the assumption.
\end{proof}

Let us return once more to the proof of \hyperlink{Proposition 3}{Proposition 3} in \S1 and specifically to equation (3). We deduced from it that for every sufficiently large $ \lambda_i $ in $ F (A_i ^ {(\nu)} (s + h_ \mu)) $, there is a term with the exponent $ \Lambda_1 + \lambda_i $. But since $F$ is of total degree $ n $, every exponent that occurs in the terms of $ F (A_i ^ {\nu} (s + h_\mu)) $ is constructed from at most $ n $ numbers via addition, where all of the numbers are equal to some number from among $ \lambda_1, \lambda_2, \ldots, \lambda_{i-1} $. We therefore get the relation $ | \Lambda_1 + \lambda_i | \leq n |\lambda_{i-1}| $ for $ \lambda_i $, from which, since $ \Lambda_1 $ is constant,
$$ \lim_ {i \to\infty} \text{sup} \frac {\lambda_i} {\lambda_{i-1}} \leq n $$
follows. We thus arrive at the theorem:
\hypertarget{Theorem 7}{}
\begin{theorem}
If a convergent Dirichlet series
$$ \sum a_ie ^ {- \lambda_i s}$$
has the property that
$$ \lim_ {i \to \infty} \text{sup} \frac {\lambda_i} {\lambda_{i-1}} \to \infty,$$
it cannot satisfy any algebraic difference differential equation.
\end{theorem}
We again return our attention to equation (3) from the proof of \hyperlink{Proposition 3}{Proposition 3}, with the added assumption that $ h_\mu = 0 $ for all such numbers, i.e. that the equation $ F(f ^ {(\nu)} ( s + h_\mu)) = 0 $ is reduced to the form $ F (f^{(\nu)} (s)) = 0 $ and therefore becomes a \textit{differential} equation. Then the expression that arises from (4) upon once again setting $ b_\varrho $ for $ b_{\varrho_t, \sigma_t} $,
\begin{align}
a_ie ^ {- (\Lambda_1 + \lambda_t) s} \sum b_{\varrho_t} (- \lambda_i) ^ {\varrho_t}
\end{align}
is contained in $ F (A_i ^ {(\nu)} (s)) $ for sufficiently large $i$. Therefore the coefficient
$$ a_i \sum b _ {\varrho_t} (- \lambda_i) ^ {\varrho_t} $$
of (27) can be expressed rationally in terms the coefficients of $F$, that is $ a_0, \ldots, a_ {i-1} $, $ \lambda_0, \lambda_1, \ldots, \lambda_ {i-1} $, using rational coefficients. But since each of the $ \lambda_i $ can be expressed as an integer linear combination of finitely many $\lambda_i$, we can represent each of the $ a_i $ in turn using the coefficients of $ F $, $ b _ {\varrho_t} $ and a finite number of the $ a_i $ and $ \lambda_i $, i.e. we may represent them rationally using \textit{finitely many numbers} and rational coefficients. - We will say that a system of numbers $ k_1, k_2, \ldots $ has \textit{a finite basis} if there are finitely many numbers $ \kappa_1, \kappa_2, \ldots, \kappa_r $ such that all $ k_i $ can be represented as rational functions of the $ \kappa_1, \ldots, \kappa_r $ and rational coefficients. With the help of this definition we can now pronouncex this theorem:
\hypertarget{Theorem 8}{}
\begin{theorem}
A Dirichlet series 
$$ \sum a_i e ^ {- \lambda_i s}, $$
whose coefficients have no finite basis cannot satisfy an algebraic differential equation.
\end{theorem}
The application of this theorem is simplified by the fact that the existence of such a basis means the numbers for a basis for the $a_i$ \textit{can be taken from the $a_i$ themselves}. Namely, 
\hypertarget{Proposition 4}{}
\begin{prop}
If a system of numbers $ k_1, k_2, \ldots $ possesses a finite basis, it also possesses a finite basis which is a subset of $\{k_i\}$.
\end{prop}

This proof requires us to check that certain eliminations can be carried out, which would prove very inconvenient if done directly using general elimination theory. However, these difficulties can be circumvented by invoking certain observations from field theory. In doing so we will rely on a few theorems from the fundamental work of E. Steinitz: \textit{Algebraic Theory of Fields}
\footnote{Crelles Journal for Mathematics, 137, (1910), pp. 167-309}, in particular from \S 22 of Steinitz's work.\\

Consider a field $ \mathscr {R} $ as the base field and in what follows only consider elements that are contained in a field $ \mathscr {Q} $ in which $\mathscr{R}$ is fixed (an ``extension'' of $ \mathscr {R } $). An element $ a $ contained in $ \mathscr {Q} $ is called \textit{algebraic} with respect to $ \mathscr {R} $ if it is the root of a single variable polynomial with coefficients from $ \mathscr {R} $. If $ \mathscr {G} $ is a system of elements, then an element $a$ is called \textit{algebraically dependent} on $ \mathscr {G} $ (with respect to $ \mathscr {R} $), if $a$ is algebraic with respect to the field $ \mathscr {R} (\mathscr {G}) $, whose construction results from adjoining the elements of $(\mathscr{G})$ to $ \mathscr {R} $. A system of elements $ \mathscr {G}'$ is called algebraically dependent on $ \mathscr {G} $ if every element contained in $ \mathscr {G}'$ is algebraically dependent on $ \mathscr {G} $. The following theorem applies: If $ \mathscr {G} _3 $ algebraically depends on $ \mathscr {G} _2 $, and $ \mathscr {G} _2 $ on $ \mathscr {G} _1 $, it follows that $ \mathscr {G} _3 $ algebraically depends on $ \mathscr {G} _1 $. Therefore we are justified in calling two systems $ \mathscr {G} _1 $ and $ \mathscr {G} _2 $ \textit{equivalent} if each algebraically depends on the other. A system $ \mathscr {G} $, which is not equivalent to any proper subset of itself, and does not consist of any elements that are algebraic with respect to $ \mathscr {R} $, is called \textit{irreducible}. We will make use of the following theorems:
\begin{enumerate}
   \item\label{item:fieldthm1} \textit{An irreducible system $ \mathscr {B} $, which algebraically depends on a finite system $ \mathscr {U} $ cannot contain more elements than $ \mathscr {U} $}\footnote {1. c. P. 292}.
    \item\label{item:fieldthm2} \textit{Let $ \mathscr {U} $ and $ \mathscr {B} $ be finite irreducible systems consisting of $ m $ and $ n $ elements respectively, and let $ \mathscr {B} $ be algebraically dependent on $ \mathscr {U}$. Then it follows that the systems $ \mathscr {U} $ and $ \mathscr {B} $ are equivalent when $m=n$, and, in the case $n<m$, that $ \mathscr{U} $ is equivalent to an irreducible system consisting of $ \mathscr {B} $ and $ m-n $ elements of $ \mathscr {U} $}\footnote {1. c. Pp. 291-292}.
\end{enumerate}

Before we go any further, let us make note of a consequence of \ref{item:fieldthm1} which we have already applied and will repeatedly use in the following sections.
\begin{lemma}
Between any $ n + 1 $ algebraic functions
$$ f_1 (x_1, \ldots, x_n), \ldots, f_ {n + 1} (x_1, \ldots, x_n) $$
of the $ n $ variables $ x_1, \ldots, x_n $,
with coefficients from a field $ K $ one can always establish an algebraic relation with coefficients from $ K $, i.e. there exists a polynomial $ \Phi (y_1, \ldots, y_ {n + 1}) $ with coefficients from $ K $, not all $ 0 $, such that 
$$ \Phi (f_1 (x_1, \ldots, x_n), f_2 (x_1, \ldots, x_n), \ldots, f_ {n + 1} (x_1, \ldots, x_n))=0$$
for indefinite $ x $\footnote {It follows immediately that $ f_1, \ldots, f_n $ are also dependent on $ x_1, \ldots, x_n $ in the usual sense as analytic functions.}.
\end{lemma}
Indeed, the system $ \mathscr {B} $ of $ n + 1 $ quantities $ f_1 (x_i), \ldots, f_ {n + 1} (x_i) $ depends algebraically on the system $ \mathscr {U} $ of the n quantities $ x_1, \ldots, x_n $ with respect to $ K $, and so therefore cannot be irreducible by result \label{item:fieldthm1}. above, given that it contains more than $ n $ quantities. It is therefore equivalent to a proper subset of itself, say
\begin{align}
f_1 (x_i), \ldots, f_ {n} (x_i).
\end {align}
Consequently, $ f_ {n + 1} (x_i) $ algebraically depends on the system (28), meaning it satisfies an equation
$$ \Phi (z, f_1, \ldots, f_n) \equiv a_0 (f_k) z ^ l + a_1 (f_k) z ^ {l-1} + \ldots + a_l (f_k) = 0, $$
in which $ a_0 (f_k), a_1 (f_k), \ldots, a_l (f_k) $ are elements from the field $ K (f_1, \ldots, f_n) $. We can now assume them to be polynomials in $ f_1, \ldots, f_n $ by multiplying $ \Phi $ with a polynomial in $ f_1, \ldots, f_n $ if necessary.
However, if $ a_0, \ldots, a_l $ are polynomials in $ f_k $, then the polynomial $ \Phi (y_1, y_2, \ldots, y_ {n + 1}) $ has the property asserted in the lemma.\\

We now return to the proof of \hyperlink{Proposition 4}{Proposition 4} and assume that the system $ (k_i) $ possesses a finite basis consisting of the numbers $ \kappa_1, \kappa_2, \ldots, \kappa_a $. 
We will use $R$ to denote the field of rational numbers and let it take the place of the field $ \mathscr {R} $ of Steinitz's theorems in what follows.
If all numbers $ \kappa $ are algebraic, then the field $ R (k_i) $ is contained in the finite extension $ R (\kappa_i) $ and therefore corresponds to one of its divisors, to be denoted as $ R_1 $.
Say $ \alpha $ is a primitive element of $ R_1 $; then $ \alpha $ must admit an expression as a finite linear combination of $ k_i $  with rational coefficients, e.g. in terms of $ k_1, k_2, \ldots, k_r $.
But since every $ k_i $ may also be expressed in terms of $ \alpha $ using rational coefficients, the numbers $ k_1, k_2, \ldots, k_r $ form a finite basis for the system $ (k_i) $ - but if not all $ \kappa_i $ are algebraic, then $ \mathscr {G}=(\kappa_i) $ must either be irreducible itself or equivalent to an irreducible subsystem $ \sum $, which consists of $b$ numbers, say $\kappa_1,\kappa_2,...,\kappa_b$.  Therefore the remaining $\kappa_i$ are algebraic with respect to $R(\kappa_1,...,\kappa_b)$, provided $a \neq b$.\\

On the other hand, the system $ (k_i) $ must also be equivalent to an irreducible system consisting of finitely many numbers $ k_i $, as long as not all $ k_i $ are algebraic. For if $ T = (k_1, k_2, \ldots, k_c) $ is an irreducible subsystem of $ k_i $, then $ T $ must be algebraically dependent on $ \sum $, and therefore by \ref{item:fieldthm1} $ c \leq b $. However, if $ T $ is chosen such that $c>0$ be as large as possible, then every $ k_i $ must be algebraically dependent on $ T $. This is the case since the system $ (T, k_i) $ consists of $ c + 1 $ elements and must therefore be equivalent to an irreducible subsystem $ T '$, with $ k_i $ being algebraically dependent on it. On the other hand, we have that $ T $ is algebraically dependent on $ (T, k_i) $, hence also on $ T '$. Thus, by \ref{item:fieldthm1}, the number of elements in $ T '$ must be equal to $ c $, and by \label{item:fieldthm2} $ T $ and $ T' $ must be equivalent. Therefore $ k_i $ is also algebraically dependent on $ T $, and $ T $ is equivalent to $ (k_i) $.\\

If not all $ k_i $ are algebraic, the system $ T $ is algebraically dependent on $ \sum $. Therefore \label{item:fieldthm2} may be applied, and the system $ \sum $ is equivalent to a system of the form ($ k_1, k_2, \ldots, k_c, y_1, y_2, \ldots, y_ {b-c} $). We will prefer to use $ x_1, x_2, \ldots, x_c $ to denote $ k_1, k_2, \ldots, k_c $, whereby these quantities need not appear at all if the $ k_i $ are all algebraic numbers, which is to say that $c=0$.
If we adjoin all $ \kappa_i $ to the field $R (x_1, x_2, \ldots, x_c, y_1, y_2, \ldots, y_ {b-c}) $, we obtain a finite extension $K$, which contains all quantities $k_i$. However, the $ k_i $ are already algebraic with respect to $ R (x_1, \ldots, x_c)$. If $ b = c $, then none of the $y$ exist, and all $ k_i $ are contained in a finite extension of $ R (x_1, \ldots, x_c) $, meaning they themselves determine a finite extension over the same field, whose primitive element we will call $ \beta $. Then $ \beta $ can be expressed rationally using finitely many $ k_i $ and $ x_1, \ldots, x_c $ with rational coefficients. On the other hand, all $ k_i $ can be expressed rationally using $ \beta $ and $ x_1, \ldots, x_c $ with rational coefficients. And since $ x_1, \ldots, x_c $ are themselves among the $ k_i $, Proposition 4 has been proven for this case.\\

Next, let $ b>c $. If we adjoin all $ k_i $ to the field $R(x_1,...,x_c)$ or, if $ c = 0 $, to the field $ R $, then we obtain an algebraic field $ K $ over $ R (x_1, \ldots, x_c) $, and it suffices to prove that $ K $ is finite with respect to $ R (x_1, \ldots, x_c) $. This owes to the fact that if true, then all $ k_i $ can be expressed rationally using a primitive element $ \gamma $ of $ K $ and $ x_1, \ldots, x_c $ with rational coefficients, and since $ x_1, \ldots, x_c $ are included among the $ k_i $, Proposition 4 follows. However, if the degree of $ K $ with respect to $ R (x_1, \ldots, x_c, y_1, \ldots, y_ {b-c}) $ is $m$, we claim that the degree of $ K $ with respect to  $ R (x_1, \ldots, x_c) $ is at most $ m $, that is to say that for any $m+1$ elements of $ K $
\begin{align}
k ', k'', \ldots, k ^{(m + 1)},
\end{align}
there is a relation between them with $ u_t (x_i) $ not all vanishing:
\begin{align}
u_1 (x_i) k'+ u_2 (x_i) k''+ \ldots + u_ {m + 1} (x_i) k ^ {(m + 1)} = 0.
\end{align}
The coefficients $ u_t (x_i) $ are $ R (x_1, \ldots, x_c) $-valued rational functions of the $ x $ with rational coefficients, that is to say rational numbers if $ c = 0 $. Since the degree of $ K $ equals $ m $ and the elements in (29) belong to the field $ K $, the relation
\begin{align}
v_1 (x_i, y_i) k'+ v_2 (x_i, y_i) k''+ \ldots + v_{m + 1} x_i, y_i) k^{(m + 1)} = 0,
\end{align}
holds, where $ v_i (x_i, y_i) $ are rational functions in the $ x_i, y_i $ with rational coefficients, not all vanishing. We can obviously assume all $ v_t $ to be polynomials in the $ x_i, y_i $. If the $ y_i $ do not show up at all in these expressions, our claim follows. But if the $ v $ also contain certain $ y $, then we rearrange (31) purely formally according to the products of powers of the $ y $. Thus we get an equation from the form
\begin{align}
\sum_j Y_j (w_1 ^ {(j)} (x_i) k'+ w_2 ^ {(j)} (x_i) k''+ \ldots + w_ {m + 1} ^ {(j)} (x_i) k ^ {(m + 1)}) = 0,
\end {align}
in which not all $ w_i ^ {(j)} (x_i) $ vanish. If there is no relation of the form (30), then the coefficients of various products of powers of the $y$ in (32) are not all $0$. If we divide (32) by the greatest common divisor of all products of powers $ Y_j $, then at least one $ y $ must actually occur in (32) following this division, since otherwise only a single $ Y $ would appear in (32) before we divide, which in turn would yield an equation of the form (30). Thus, there exists an identically vanishing polynomial in $ y_1, \ldots, y_ {b-c} $, whose coefficients are not all 0 and algebraically dependent on $ x_1, \ldots, x_c $. But this contradicts the assumption that the system $ x_1, \ldots, x_c $, $ y_1, \ldots, y_ {b-c} $ is irreducible. With that Proposition 4 is proven in its entirety\footnote {
The fact proved in the last part of the proof is essentially identical to a proposition proven by Mr. I. Schur, Berlin proceedings report. (1911), p. 619 ff.
}.\\

This proposition produces the following stronger version of \hyperlink{Theorem 8}{Theorem 8}:
\begin{corollary}
If there is a subsystem $ \mathscr {G} $ of the coefficients $ a_i $ of a Dirichlet series
\begin{align}
\sum a_i e ^ {- \lambda_i s},
\end {align}
whose numbers cannot be expressed rationally (with rational coefficients) using finitely many numbers belonging to the same subsystem, then this series cannot satisfy any algebraic differential equation.
\end{corollary}
In particular, (33) does not satisfy an algebraic differential equation if the numbers $ a_i $ contain an infinite number of algebraic numbers not contained in the same finite field extension, for instance an infinite number of distinct roots of unity.

\newpage

\section{On algebraic-transcendent functions.}

In order to expand on the results established in the previous sections and subsequently apply them to power series, we must first delve into general properties of single variable functions that satisfy algebraic differential equations. Following E.H. Moore, we will refer to functions satisfying an algebraic differential equation as \textit{algebraic-transcendent functions}\footnote {Math. Ann. \textbf {48} (1897) p. 49. Mathematical journal. VIII.}.\\

In this section, a \textit{domain of rationality} $ R $, will be understood to mean a field of analytic functions sharing a common domain $ G(R) $, on which each of them is regularly analytic except at isolated points. In addition, we will assume that for every function in $R$, all of its derivatives are also elements of $R$. If a function $ f (x) $ is defined on $ G (R) $ and satisfies a differential equation
$$ \Phi (f ^ {(i)}) = 0 $$
where $ \Phi $ is a polynomial in $ f $ and its derivatives whose coefficients are elements in $ R $, we will say that it is \textit{algebraic-transcendent with respect to $ R $}. \\

Let $ g (x) $ be an algebraic-transcendent function with respect to $ R $ and let $ f (x) $ an algebraic-transcendent function with respect to the domain of rationality $ R (g) $, which results from adjoining $g$ as well as all its derivatives to $R$. We claim \textit{that $ f (x) $ is also algebraic-transcendent with respect to $ R $}. Let $ g (x) $ satisfy a differential equation $ m $-th order
\begin{align}
 \Phi  (g ^{(i)}) = 0,
\end{align}
whose left-hand side is a polynomial in $ g, g', \ldots, g^{(m)} $ with coefficients from $ R $.  Moreover, suppose $g$ satisfies no such differential equations of lower order or of the same order and of lower degree in $ g ^ {( m)} $. Then, by repeated differentiation of equation (34), we may represent each of the higher order derivatives $ g ^ {(m + 1)}, g ^ {(m + 2)}, \ldots $ rationally using  $ g, g ', \ldots, g ^ {(m)} $ and coefficients from $ R $, where the denominator contains a power of $ S = \frac {\partial \Phi} {\partial g ^ {(m)}} $ , with this last expression clearly not vanishing identically. Now suppose $ f (x) $ satisfies a differential equation of $ n $-th order

\begin{align} \psi (f ^ {(i)}) = 0, 
\end{align}
where $\psi (f ^ {(i)}) $ is a polynomial in $ f, f', \ldots, f ^ {(n)} $ with coefficients from $ R (g) $, and satisfies no equations of lower order or $n$-th order and of lesser degree in $ f ^ {(n)} $. Therefore $ T = \frac {\partial \psi} {\partial f ^ {(n)}} $ is certainly non-zero. By repeated differentiation of equation (35) one can represent all higher order derivatives $ f ^ {(n + 1)}, f ^ {(n + 2)}, \ldots $ rationally using $ f, f', \ldots, f ^ {(n)} $ and $ g, g', \ldots $, where only a power of $T$ occurs as the denominator. If one substitutes $ g ^ {(m + 1)}, g ^ {(m + 2)}, \ldots $ for the expressions we obtained above, one can represent all derivatives $ f ^ {(n + 1)}, f ^ {(n + 2)}, \ldots $ rationally using $ f, f ', \ldots, f ^ {(n)} $, $ g, g', \ldots, g ^ {(m)} $ with coefficients from $ R $, where the denominator consists only of products of powers of $T$ and $S$. The subsequent application of Lemma \ref{lemma:sec3} from \S3 proves the existence of a rational relation with coefficients from $R$ between a sufficiently large number of derivatives of $ f $. This proves our claim. \\

Next, let $ g_1, g_2, \ldots, g_k $ be algebraic-transcendent function with respect to $ R $, and $ f $ be algebraic-transcendent with respect to $ R (g_1, g_2, \ldots, g_k) $. Then, by the just proven claim, $ f $ is algebraic-transcendent with respect to $ R(g_2, g_3, \ldots, g_k) $, hence also with respect to $ R (g_3, \ldots, g_k) $ etc., hence also with respect to $ R $ itself. It thus follows that:

\hypertarget{Theorem 9}{}
\begin{theorem}
If an analytic function $ f (x) $ satisfies a differential equation
$$ \Phi (f ^ {(i)}) = 0, $$
whose left-hand side is a polynomial in $ f, f', f'', \ldots $, with coefficients that are algebraic-transcendent with respect to a domain of rationality $ R $, then $ f (x) $ is itself algebraic-transcendent with respect to to $ R $, provided $ f (x) $ is defined on the domain of definition for $ R $.
\end{theorem}

This last detail is necessary because $ f (x) $ could be separated from $ G (R) $ by a natural boundary, while the common domain of existence for the coefficients of $ \Phi $ could simultaneously contain $ G (R) $ as well as the domain where $ f (x) $ is defined. \\

But additionally, it follows from the observation above \textit{that every rational function of algebraic-transcendent functions with respect to $ R $ is itself algebraic-transcendent with respect to $ R $}\footnote {Result due to Stadigh (1. c. P. 5)}.
And more generally, the same is true for every \textit{algebraic} function of algebraic-transcendent functions with respect to $R$.\\

We now specialize  by letting $R$ be the field of all constants, such that from now on we will only consider algebraic-transcendent functions pure and simple. Let $ g (x) $ and $ f (x) $ be algebraic-transcendent functions, and let the range of $ g (x) $ and the domain of $ f (x) $ have a region in common. Let us consider the function $ f (g (x)) = \varphi (x) $. If $ g (x) $ is not constant, the functions $ f ^ {(i)} (g (x)) $, which are the result of plugging $g(x)$ into the derivatives of $ f (x) $, can be represented rationally using derivatives of $\varphi(x)$ and $g(x)$, where the denominator consists only of a power of $ g'(x) $. 
Thus, by \hyperlink{Proposition 1}{Proposition 1}, $ f (x) $ satisfies an algebraic differential equation with constant coefficients not explicitly depending on $ x $. Considering this equation and replace $ f (x), f'(x), f''(x), \ldots $ with the expressions for $ f (g (x)), f' (g (x)), f''(g (x)), \ldots $ in terms of the derivatives of $ \varphi (x) $ and of $ g (x) $.  This substitution gives rise to a rational relation between the derivatives of $ \varphi (x) $ and $ g (x) $, which does not hold identically since for $ g (x) = x $, $ \varphi (x) = f (x) $, it turns into our initial differential equation for $ f (x) $ again. From \hyperlink{Theorem 9}{Theorem 9}, it therefore follows that $ \varphi (x) = f (g (x)) $ \textit{itself is an algebraic-transcendent function}. \\

This fact gives us the opportunity to construct analytic functions of several variables, which have the property that if one takes the variables to be algebraic-transcendent functions of $ x $, we once again have an algebraic transcendent function. Say we construct an arbitrary algebraic function using algebraic-transcendent functions of arbitrary variables $ x, y, z, \ldots $. Now plug functions obtained this way into arbitrary algebraic-transcendent functions as variables, from these functions consider any number of algebraic functions, etc. That is to say, we are repeating the process which in \S3 led to the construction of improper functions of two variables, with the only difference being that instead of combining \textit{arbitrary analytic} functions of one variable and algebraic functions of two variables, we use\textit{ algebraic-transcendent functions} of one variable and algebraic functions of \textit{several} variables. Every such function $ f (x, y, z, \ldots) $ has, according to our theorems, the property that $ f(\varphi (x) , \psi (x), \chi (x), \ldots) $ becomes algebraic-transcendent as soon $ \varphi (x), \psi (x), \chi (x), \ldots $ themselves are. In the next section we shall construct the most general kind of function possessing this property. We also observe that if such a function $ f (x, y, z, \ldots) $ can be expanded in terms of positive integer powers of $ x, y, z, \ldots $, especially as the Dirichlet series
$$ f (e ^ {- \lambda_1s}, e ^ {- \lambda_2s}, \ldots) $$
for positive $ \lambda_1, \lambda_2, \ldots $, it is also an algebraic-transcendent function. \\

From every power series
$$ f (x) = a_0 + a_1x + a_2x ^ 2 + \ldots, $$
which converges in a neighborhood of the origin and represents an algebraic-transcendent function, a Dirichlet series results from the substitution $ x = e ^ {- s} $ which (possesses a domain of absolute convergence and in that domain) represents an algebraic-transcendent function. We can therefore adapt \hyperlink{Theorem 7}{Theorem 7} and \hyperlink{Corollary 1}{Corollary 1} to ordinary power series. From Corollary 1 we find

\hypertarget{Theorem 10}{}
\begin{theorem}
If for a power series
$$ f (x) = a_0 + a_1 x + a_2 x ^ 2 + \ldots, $$
converging in a neighborhood of the origin, one can find a subsystem of coefficients whose elements cannot be represented rationally by a finite number of them, with rational coefficients, then the function $ f (x) $ represented by this power series is not algebraic-transcendent.
\end{theorem}

This means, for example, that the series
\begin{align} \zeta (x, \nu) = \frac {x} {1 ^ \nu} + \frac {x ^ 2} {2 ^ \nu} + \frac {x ^ 3} {3 ^ \nu} + \ldots \end {align}
does not satisfy any algebraic differential equation (as a function of $x$) for any $\nu$ a non-integer rational number. 
This follows from the observation that the algebraic numbers $ 2 ^ \nu, 3 ^ \nu, \ldots $ cannot all be found in a finite field extension, since otherwise, for instance, the discriminant of this number field would have to be divisible by the discriminant of each sub-field, i.e., as per Landsberg's theorems\footnote {Crelles Journal for Mathematics, \textbf {117} (1987), pp. 140 ff.}, every prime number.
For $ \nu $ an integer this series always satisfies an algebraic differential equation. For irrational $ \nu $ it is presumably never the case, but I have been unable to prove it.
From the theorems of the next section it follows, however, that the set of $ \nu $, for which function (36) is algebraic-transcendent, is countable.
On the other hand, as Mr. G. Pólya brought to my attention, and with his friendly permission I would like to point out here, is that it is easy to show that $ \zeta (x, \nu) $ is not algebraic for irrational $ \nu $ , if one considers the behavior of $ \zeta (x, \nu) $ at $ x = 1 $. 
Using the functional equation, it suffices to take $ 0 <\nu <1 $.
But then (see E. Picard, Traité d'Analyse, \textbf {1.} (2 éd.), P. 230) $ \lim_ {x \to 1} (1-x) ^ {1 - \nu} \zeta (x, \nu) $ equals $ \Gamma (1- \nu) \neq0 $, as $ x $ approaches the points $x = 1$ radially from inside the unit circle, that is to say $ \zeta (x, \nu) $ becomes infinite for $ x = 1 $ of an irrational order. \\

From \hyperlink{Theorem 7}{Theorem 7} it follows that a power series
$$ a_1x ^ {n_1} + a_2x ^ {n_2} + \ldots, $$
in which the actually occurring non-negative integer exponents $ n_1, n_2, \ldots $ grow so rapidly that $ \overline{\lim_{i \to \infty}} \frac {n_i}{n_{i-1}} = \infty $, cannot satisfy an  algebraic differential equation. 
This result is well known and is due to Grönwall\footnote {Öfversigt af Vetensk. Akademiens Förhandlingar, \textbf {55} (Stochholm, 1898). P. 387 ff.}. 
The following theorem may also be deduced from it, which is analogous to a well-known theorem by G. Pólya about non-continuable power series\footnote {Acta Mathematica, \textbf {40} (1916), pp. 179-183.}:

\hypertarget{Theorem 11}{}
\begin{theorem}
From every infinite power series converging in a neighborhood of the origin, one can obtain infinitely many power series that do not represent any algebraic-transcendent function by multiplying certain coefficients by - 1 
\end{theorem}

Indeed, let
$$ Q (x) = a_ {n_1} x ^ {n_1} + a_ {n_2} x ^ {n_2} + \ldots $$
be a power series constructed from certain terms of
$$ P (x) = a_0 + a_1 x + a_2 x ^ 2 + a_3 x ^ 3 + \ldots $$
in such a way that it is not an algebraic-transcendent function due to the growth of its exponents. If $ P (x) $ satisfies an algebraic differential equation, the power series $ P_1 (x) $ resulting from it by changing the signs of all terms occurring in $ Q (x) $ cannot satisfy any algebraic differential equation, otherwise the same would also apply to $ Q (x) = \frac{1}{2} (P (x) -P_1 (x)) $.
But if $ P (x) $ is not algebraically transcendent, and there is at least one algebraically transcendent power series among
$$ \pm a_0 \pm a_1x \pm a_xx ^ 2 \pm a_3 x ^ 3 \pm \ldots $$
we can apply our conclusion to it and choose $ Q (x) $ in an infinitely many different ways. - From a Dirichlet series
$$ \sum a_i e ^ {- \lambda_i s} $$
the substitution $ e ^ {- s} = x $ creates a series containing infinitely growing powers of $ x $
$$ \sum a_i x ^ {\lambda_i}, $$
in which the exponents $ \lambda_i $ no longer need to be integers, much less even rational. Since such series expansions appear quite often in the investigation of the behavior of integrals of ordinary differential equations, it is of interest to formulate the following consequence for such series arising from \hyperlink{Theorem 6}{Theorem 6}:

\hypertarget{Theorem 12}{}
\begin{theorem}
If the series
$$ \sum a_i x ^ {\lambda_i} $$
converges in a certain neighborhood of the origin
and the exponents $ \lambda_i $ are real numbers growing infinitely, having no finite linear basis, then the analytic function represented by this series cannot satisfy any algebraic differential equation.
\end{theorem}

With certain restrictions, the theorem can also be generalized to the case that $ \lambda_i $ take on complex values. In this case, however, direct evidence will emerge from the investigations within the last section. \\

We will understand (herein and in the sequel) $ x ^ \lambda $ to always be $ e ^ {\lambda \text{log}x}$ where the coefficient of $ \sqrt {-1} $ in $ \text{log}x $ is taken to be non-negative and less than $ 2\pi $.

\newpage

\section{On algebraic-transcendent function of several variables.}

In the previous section we alluded to a certain class of analytic functions of several variables $ f (x, y, z, \ldots) $ having the following property: If one replaces $ x, y, z, \ldots $   in $ f (x, y, z, \ldots) $ with any algebraic-transcendent functions $ \varphi_1 (s), \varphi_2 (s), \varphi_3 (s), \ldots $ of a variable $ s $, then the resulting function
$$ f (\varphi_1 (s), \varphi_2 (s), \varphi_3 (s), \ldots) $$
is again an algebraic-transcendent function, provided the domains and range of the $\varphi$ and $f$ respectively line up properly. We now take it upon ourselves to construct all functions of several variables with this property. We will refer to this property as property $ A $ in short. Let $ f (x, y, z) $ be such a function. Then $ f (x, y, z) $ is algebraic-transcendent as a function of $x$ for every fixed value of $ y, z $ . Next, consider the set of all products of powers of an indeterminate function $ \Phi (t) $ and its derivatives of arbitrary order. This set is countable, since for every product of powers
\begin{align}
\Phi (t) ^ \alpha \Phi '(t) ^ \beta \ldots \Phi^{(i)}(t) ^ \gamma \ldots
\end {align}
the sum $ \alpha + 1 \beta + \ldots + i \gamma + \ldots $ is finite and for the value of this sum only finitely many such products exist. Furthermore, the set $ U $ of all finite subsets of the set of products (37) is also countable, as is known and immediately verifiable. For every value of $ y, z $ there is now a system of products of powers:
\begin{align}
\Pi_1 (\Phi ^ {(i)}), \Pi_2 (\Phi ^ {(i)}), \ldots, \Pi_k (\Phi ^ {(i)})
\end {align}
of the form (37) and, with the property that for said $ y, z $, we have
\begin{align}
c_1 \Pi_1 (f ^ {(i)} (x, y, z)) + c_2 \Pi_2 (f ^ {(i)} (x, y, z)) + \ldots + c_k \Pi_k (f ^ { (i)} (x, y, z)) = 0,
\end{align}
where the derivatives of $ f (x, y, z) $ are only taken with respect to $ x $, and $ c_1, c_2, \ldots, c_k $ are constants which do not all vanish. Now, set $ z = z_0 $ and let $ y $ take values in an arbitrarily small strip $ a \leq y \leq b $ in the corresponding domain of regularity ($ a \leq y \leq b $, $ c \leq x \leq d $) of $ f $. Then $y$ takes on values in a set having the magnitude of the continuum. And since the magnitude of the set $ U $ is smaller, there is at least one system of products of powers (38) for which there is a relation of the form (39) for infinitely many $ y $. The coefficients $ c_1, \ldots, c_k $ can naturally still depend on the respective value of $y$. Let us now construct the Wronskian determinant of the functions
\begin{align}
\Pi_1 (f ^ {(i)} (x, y, z_0)), \Pi_2 (f ^ {(i)} (x, y, z_0)), \Pi_k (f ^ {(i)} (x , y, z_0))
\end{align}
with respect to $ x $, giving rise to an analytic function $ F (x, y) $ which vanishes identically in $ x $ for infinitely many $ y $ in interval $ a \leq y \leq b $. Therefore, for $x$ fixed, it has an infinite number of roots in the interval $ a \leq y \leq b $ as a function of $ y $,  meaning it vanishes identically in $ x, y $ for each fixed value of $z$. Therefore, there exists a relation of the form
\begin{align} c_1 (y) \Pi_1 (f ^ {(i)} (x, y, z_0)) + c_2 (y) \Pi_2 (f ^ {(i)} (x, y, z_0)) + \ldots + c_k (y) \Pi_k (f ^ {(i)} (x, y, z_0)) = 0,
\end{align}
between the functions ( 40), in which, as the well-known proof of the theorem on the Wronskian determinant immediately yields, the $ c_1 (y), c_2 (y), \ldots, c_k (y) $ can be taken as analytic functions of $ y $ for $ a \leq y \leq b $ which do not all vanish entirely. We then differentiate the relation (41) $(k-1)$ times with respect to $x$. This leaves us with an additional $ k-1 $ equations of the form
\begin{align}
c_1 (y) \Pi_1 ^ {(\nu)} (f ^ {(i)} (x, y, z_0)) + c_2 (y) \Pi_2 ^ {(\nu)} (f ^ {(i) } (x, y, z_0)) + \ldots 
\end {align}
$$+ c_k (y) \Pi_k ^ {(\nu)} (f ^ {(i)} (x, y, z_0)) = 0$$
for $ \nu = 1,2, \ldots, k-1 $. Here $ \Pi_l ^ {(\nu)} (\Phi ^ {(i)}) $ denotes the differential expression obtained by differentiating $ \Pi_l (\Phi ^ {(i)}) $ $ \nu $ times with respect to the independent variable. Let us consider the determinant
$$
\begin {vmatrix}
\Pi_1 (\Phi ^ {(i)}) & \Pi_2 (\Phi ^ {(i)}) & \ldots & \Pi_k (\Phi ^ {(i)}) \\
\Pi_1 ^ {(1)} (\Phi ^ {(i)}) & \Pi_2 ^ {(1)} (\Phi ^ {(i)}) & \ldots &
\Pi_k ^ {(1)} (\Phi ^ {(i)}) \\
\Pi_1 ^ {(k-1)} (\Phi ^ {(i)}) & \Pi_2 ^ {(k-1)} (\Phi ^ {(i)}) & \ldots &
\Pi_k ^ {(k-1)} (\Phi ^ {(i)}) \\
\end {vmatrix}
= T (\Phi ^ {(i)}).
$$
This determinant represents a \textit{non-identically vanishing} differential expression of $ \Phi $. Were it to vanish identically, by substituting the Riemann $ \zeta $-function $ \zeta (s) $ for $ \Phi $, it would result in the Wronskian determinant of the $ k $ functions $ \Pi_1 (\zeta ^ {(i)} (s)), \Pi_2 (\zeta ^ {(i)} (s)), \ldots, \Pi_k (\zeta ^ {(i)} (s)) $.  Vanishing would imply the existence of the relation
\begin{align}
c_1 \Pi_1 (\zeta ^ {(i)} (s)) + c_2 \Pi_2 (\zeta ^ {(i)} (s)) + \ldots + c_k \Pi_k (\zeta ^ {(i)} ( s)) = 0
\end {align}
with $ c $ not all vanishing. Since $ \Pi_1, \Pi_2, \ldots, \Pi_k $ are all products of the function $ \Phi $ and its derivatives that differ from one another, (43) results in an algebraic differential equation that is satisfied by $ \zeta (s) $. However, this contradicts Hilbert's theorem that $ \zeta (s) $ does not satisfy any algebraic differential equation. \\

 So, for every value $ z_0 $ of $ z $, $f(x,y,z)$ satisfies an algebraic differential equation in $ x $ with constant coefficients:
$$ T (f ^ {(i)} (x, y, z_0)) = 0. $$

We can now let $z$ vary, and by repeating the same conclusions we arrive at a relation of the form

\begin{align}
c_1 (y, z) \overline {\Pi_1} (f ^ {(i)} (x, y, z)) + c_2 (y, z) \overline {\Pi_2} (f ^ {(i)} (x, y, z)) + \ldots
\end {align}
$$ + c_ {k '} (y, z) \overline {\Pi_ {k'}} (f ^ {(i)} (x, y, z)) = 0, $$
in which $ \overline {\Pi_1} (\Phi ^ {(i)}), \ldots, \overline {\Pi_ {k '}} (\Phi ^ {(i)})$ are once again products of powers of the function $ \Phi $ and its derivatives. By now differentiating (44) in $x$ $ (k'-1) $ times and following the same process as above, we obtain an algebraic differential equation (in $ x $), which $ f (x, y, z) $ satisfies identically in $ x, y, z $. This proof evidently does not rely on the number of variables. We formulate this result as
\hypertarget{Proposition 5}{}
\begin {prop}
If an analytic function $ f (x, y, z, \ldots) $ has the property that it satisfies an algebraic differential equation (as a function of $ x $) in the neighborhood of a regular point for every value of $ y, z, \ldots $, it also satisfies a constant coefficient algebraic ordinary differential equation as a function of $ x $ identically in $ x, y, z, \ldots $.
\end {prop}
In the case of two variables $ x, y $, our proof says more. We have only used the fact that the magnitude of the set of $ y $ -values for which $ f (x, y) $ is algebraic-transcendent is greater than the magnitude of a countable set. And from this we can conclude that $ f (x, y) $ is algebraic-transcendent for every value of $ y $ for which the function $ f (x, y) $ remains analytic in $ x $. This results in
\hypertarget{Theorem 13}{}
\begin{theorem}
If there exists even a single fixed value of $ y $ for which an analytic function $ f (x, y) $ remains analytic in $ x $ and does not satisfy any algebraic differential equation with respect to $ x $, then there are at most a countable number of fixed values of $ y $, for which $ f (x, y) $ remains analytic in $ x $ and is algebraic-transcendent.
\end{theorem}

In particular, we see that the power series
$$ \frac {x} {1 ^ \nu }+ \frac {x ^ 2} {2 ^ \nu} + \frac {x ^ 3} {3 ^ \nu} + \ldots $$
only satisfies an algebraic differential equation as a function of $ x $ for a countable number of fixed $ \nu $, since it is not algebraic-transcendent for rational non-integer $ \nu $ (for every integer $ \nu $, on the other hand, it is algebraic-transcendent, as one can easily verify). \\

From this we obtain the following more strong result:

\hypertarget{Theorem 14}{}
\begin{theorem}
If for every positive $ M $ there exists a $ y $ such that the analytic function $ f (x, y) $ remains analytic for this fixed $ y $ and does not satisfy any algebraic differential equation whose order is less than $ M $, then there is there are only countably many $ y $ for which $ f (x, y) $ remains analytic in $ x $ and is algebraic-transcendent.
\end{theorem}
Let us now return to our task. From our premises we can conclude that $ f (x, y, z, \ldots) $ satisfies an ordinary algebraic differential equation with constant coefficients with respect to each of its variables
\begin{align}
\Phi_x (f ^ {(i)} (x, y, z, \ldots)) = 0, \Psi_y (f ^ {(i)} (x, y, z, \ldots)) = 0,
\end {align}
$$ X_z (f ^ {(i)} (x, y, z, \ldots)) = 0, \ldots $$

From this it follows that every partial differential quotient of $ f $ of a certain order $ m + 1 $ can be expressed algebraically by $ f $ and its derivatives up to $ m $-th order. Let $ \Phi_x = 0 $ have order $ \alpha $, $ \Psi_y = 0 $ have order $ \gamma $ etc. Then take $m$ to be $ \alpha + \beta + \gamma + \ldots -1$.  Now it suffices to prove that every derivative
\begin{align}
\frac {\partial ^ {a + b + c + \ldots} f} {\partial x ^ {a} \partial y ^ b \partial z ^ c \ldots},
\end {align}
for $ a + b + c + \ldots \geq \alpha + \beta + \gamma $ can be expressed algebraically using partial derivatives of $ f $ up to the order $ a + b + c + \ldots -1$.  Since  $ a + b + c + \ldots \geq \alpha + \beta + \gamma + \ldots $, it follows that either $ a \geq \alpha $, or $ b \geq \beta $, or $ c \geq \gamma $, $\ldots$. Suppose $ a \geq \alpha $. We thus construct the differential equation
\begin{align}
\frac {\partial ^ {a + b + c + \ldots} \Phi_x (f ^ {(i)}) (x, y, z, \ldots)} {\partial x ^ {a} \partial y ^ b \partial z ^ c \ldots} = 0.
\end {align}

In it the partial derivative (46) occurs linearly, multiplied with
\begin{align}
\frac {\partial \Phi_x} {\partial f ^ {(\alpha)}},
\end {align}
where $ f ^ {(\alpha)} $ is the $ \alpha $-th derivative in $ x $ of $f$. And since we can assume that (45) are equations of the lowest possible order and among those are also of the lowest possible total degree, (48) is non-zero. On the other hand, all derivatives occurring in (47), with the exception of (46), are of order less than $ a + b + c + \ldots, $ whence our claim is proven.\\

A system of partial differential equations for a function $ f (x, y, z, \ldots) $, which allows all differential quotients of a certain order to be expressed in terms of the differential quotients of lower order is called a \textit{Mayer system}. If the left-hand sides of the differential equations in a Mayer system are polynomials in the derivatives, whose coefficients are algebraic in the independent variables, we have what is called an \textit{algebraic} Mayer system. \textit{If a function $ f (x, y, z, \ldots) $ satisfies an algebraic Mayer system, then, for each of its variables, it always satisfies an ordinary differential equation with constant coefficients}. This holds since by definition, each of the partial derivatives $ \frac {\partial ^ \alpha f (x, y, z, \ldots} {\partial x ^ \alpha} $ can be expressed algebraically by finitely many partial derivatives of $ f $ and by $ x, y, z, \ldots $; but according to Lemma \ref{lemma:sec3} of \S3, it follows from this that a rational equation with constant coefficients holds between a sufficiently large number of partial derivatives $ \frac {\partial ^ \alpha f (x, y, z, \ldots )} {\partial x ^ \alpha} $. From this and the result above, it can be further established that every function satisfying an algebraic Mayer system, also satisfies such a Mayer system whose differential equations on the left-hand side are polynomials in the derivatives with \textit{constant} coefficients.- A Mayer system can also be characterized in a different fashion. The general integral of such a system does not depend on arbitrary functions, but only on a finite number of arbitrary constants. Conversely, via an oft-cited theorem, any system of differential equations whose general integral depends only on finitely many arbitrary constants is a Mayer system. However, a rigorous proof of this last theorem is not to be found in the literature. It can only initially be derived from Riquier's investigations, which establish a normal form for systems of partial differential equations, off of which the number and type of arbitrary elements in the general integral can be read directly, and into which every system of partial differential equations can be converted using differentiation and elimination\footnote {Riquier, Les systèmes d'équations aux dérivées partielles. Paris, Gauthier-Villars, 1910.}.
\\
The result obtained above can now be stated as follows:  every function possessing property $A$ satisfies an algebraic Mayer system.  This result can be reversed, from which we conclude that

\hypertarget{Theorem 15}{}
\begin{theorem}
If an analytic function $ f (x, y, z, \ldots) $ has the property of always being an algebraic-transcendent function of $ s $ as soon $ x, y, z, \ldots $ are taken to be algebraic-transcendent functions of $ s $, whose ranges and domains line up with one another and the domain of $ f (x, y, z, \ldots) $, - or additionally only if all variables $ x, y, z, \ldots $ are replaced with arbitrary numbers from certain uncountable sets, with the exception of one which is set equal to $ s $, then $f(x,y,z,\ldots)$ satisfies an algebraic Mayer system and vice-versa.
\end{theorem}

The proof of the reversed statement presents no difficulties. We will proceed to link this to some more general and farther reaching observations.
\\
Indeed, the analytic functions that satisfy algebraic Mayer systems constitute the proper generalization of single-variable algebraic-transcendent functions. In a certain sense these form a closed system since we can state entirely analogous results for them as we did in previous sections. We will designate them as \textit{algebraic-transcendent functions of several variables}. To further generalize this concept, we define the \textit{domain of rationality} $R$ as a field of analytic functions of a chosen set of variables $ x, y, z, \ldots $ which all share a common domain $ G (R) $ in which each of them is regular analytic save for $ (n-1) $ - dimensional singular manifolds. Moreover, we require that $ R $ also contains the partial derivatives of all orders of any function of $ x, y, z, \ldots $ contained in it. We now call an analytic function $ f (x, y, z, \ldots) $ \textit{algebraic-transcendent with respect to $ R $} if it satisfies a Mayer system, whose differential equations result from setting polynomials in $ f $ and their derivatives equal to zero, where the coefficients are taken from $R$. \\

If $ R $ is the field of all constants, $ f (x, y, z, \ldots) $ becomes quintessentially algebraic-transcendent. Just as above, we can show:
\begin {enumerate}
\item Every $ f (x, y, z, \ldots) $ that is algebraic transcendent with respect to $ R $ satisfies an ordinary differential equation in each of its variables, the left-hand side of which is a polynomial in $ f $ and its derivatives in the appropriate variable, whose coefficients are taken from $ R $.
\item If an analytic function $ f (x, y, z, \ldots) $ satisfies an ordinary differential equation in each of its variables, whose left side is a polynomial in $ f $ and its derivatives with respect to that variable, with coefficients from $ R $, then $ f $ is algebraic-transcendent with respect to $ R $.
\end {enumerate}
And now the following theorems apply:

\hypertarget{Theorem 16}{}
\begin{theorem}
If $ f (x, y, z, \ldots) $ is algebraic-transcendent with respect to the domain of rationality $ R (g, h, \ldots) $, which arises from a domain of rationality $ R $ via the adjunction to $R$ of the functions $ g, h, \ldots $ and all their derivatives , which are each algebraic transcendent with respect to $R$. Then $ f (x, y, z, \ldots) $ is also algebraic-transcendent with respect to $ R $.
\end{theorem}
\begin{proof}

For this proof, it suffices to consider only the case where a \textit{single} function $ g (x, y, z, \ldots) $ and its derivatives are adjoined to $ R $. The more general result will then follow through repeated application of this case to the case where only \textit{finitely many} functions $ g, h, \ldots $ along with their derivatives are adjoined to $R$. On the other hand, all differential equations belonging to the Mayer system satisfied by $ f $ with respect to $ R (g, h, \ldots) $, can be derived from a finite number of differential equations, for instance from the ordinary differential equations which are satisfied by $f$ in each individual variable. Therefore we are essentially only dealing with finitely many elements from $R(g,h,...)$, along with their derivatives. \\

If all partial derivatives of $ f $ can already be expressed algebraically using only $ \mu $ of them along with elements from $ R (g) $, and all derivatives of $ g $ can also be expressed using only $ \nu $ of them along with elements from $ R $, then an entire rational relation with coefficients from $R$ must hold between some $\mu + \nu + 1$ many partial derivatives of $f$ by Lemma \ref{lemma:sec3} from \S3. Next, let $ \lambda $ be the smallest number with the property that between some $ \lambda $ many partial derivatives of $ f $ there exists an entire rational relation with coefficients from $ R $, and let us say $ p_1, p_2, \ldots , p_{\lambda-1} $ are $ \lambda-1 $ partial derivatives of $ f $ for which such a relation does not exist. Then each partial derivative of $ f $ is an algebraic function of $ p_1, p_2, \ldots, p _ {\lambda-1} $ and elements from $ R $, as we wished to show\footnote {For analytic functions $ F (x, y, z, \ldots) $, for which a single partial differential equation is prescribed, one can prove an analog of this theorem much the same way without additional obstacles: If $ F (x, y, z, \ldots) $ satisfies a partial differential equation, the left side of which is a polynomial on the derivatives of $ F $ whose coefficients are algebraic-transcendent functions of several variables, then $ F $ satisfies an algebraic partial differential equation.}.
\end{proof}

\hypertarget{Theorem 17}{}
\begin{theorem}
If one replaces the variables $ x, y, z, \ldots $ of an algebraic-transcendent function $ f (x,y, z, \ldots) $ with arbitrary algebraic transcendent functions $ \varphi (\xi, \eta, \zeta, \ldots ),$
$ \psi (\xi, \eta, \zeta, \ldots), \chi (\xi, \eta, \zeta, \ldots), \ldots $ of arbitrary variables $ \xi, \eta, \zeta, \ldots $, we are still left with an algebraic-transcendent function of $ \xi, \eta, \zeta, \ldots $
\end{theorem}

In particular, this theorem contains the converse of \hyperlink{Theorem 15}{Theorem 15} as a special case. It is only formulated for \textit{quintessentially algebraic-transcendent functions}. - If one wants to generalize this to algebraic-transcendent functions in relation to an arbitrary domain of rationality, it becomes necessary to restrict our definition of a domain of rationality $ R $. We additionally have to require that for every function in $ R $, by pre-composing with any function in $ R $, we either give rise to another function in $R$, or at least one that is algebraically-transcendent with respect to $ R $.

\begin{proof}
In the proof of this theorem we will use $ \mu $ to denote the smallest number for which each derivative of $ f (x, y, z, \ldots) $ admits a constant-coefficient algebraic expression using only $ \mu $ of them.  $ \nu_1, \nu_2, \nu_3, \ldots $ we will denote the corresponding numbers for $ \varphi (\xi, \eta, \zeta, \ldots), \psi (\xi, \eta, \zeta, \ldots), \chi (\xi, \eta, \zeta, \ldots), \ldots $ If we replace $x,y,z\ldots$ in all derivatives of $ f (x, y, z, \ldots) $ with the functions $ \varphi, \psi, \chi, \ldots $, the resulting functions of $ \xi, \eta, \zeta, \ldots $ will be denoted by $ p_i $. The $ p_i $ also have the property that between any $ \mu + 1 $ of the functions $ p_i $ there exists an entire rational relation with numerical coefficients. Next, let $ \lambda $ be the smallest number with the property that between $ \lambda + 1 $ of the functions $ p_i $ there is always an entire rational relation with numerical coefficients. Then all $ p_i $ can be expressed algebraically by only $\lambda$ many of them. Every partial derivative of $ f (\varphi, \psi, \chi, \ldots) $ is an entire rational function of the $ p_i $ and the derivatives of $ \varphi, \psi, \chi, \ldots $.  We may thus represent it as an algebraic function of $ \lambda + \nu_1 + \nu_2 + \nu_3 + \ldots $ many quantities. Therefore, according to Lemma \ref{lemma:sec3} from \S3, there exists an entire rational relation with numerical coefficients between $ \lambda + \nu_1 + \nu_2 + \nu_3 + \ldots + 1 $ derivatives of $ f (\varphi, \psi, \chi, \ldots) $. Therefore, all of these derivatives may be expressed algebraically using finitely many (at most $ \lambda + \nu_1 + \nu_2 + \nu_3 + \ldots $) of them, as required.
\end{proof}

\newpage

\section{On algebraic-transcendent Dirichlet series and power series whose terms include arbitrary powers of the independent variables.}

If the function represented by the Dirichlet series

\begin{align}
f (s) = \sum_ {i = 0} ^ \infty a_i e ^ {- \lambda_i s}
\end{align}
is algebraic-transcendent, by \hyperlink{Theorem 1}{Theorems 1} and\hyperlink{Theorem 6}{ 6}, the series (49) can be rewritten as
\begin {align}
F (e ^ {- \lambda ^ {(1)} s}, e ^ {- \lambda ^ {(2)} s}, \ldots, e ^ {\lambda ^ {(\mu)} s}),
\end{align}
where $ F (x, y, z, \ldots) $ is an initially formal power series whose terms contain both positive and negative integer powers of $ x, y, z, \ldots $. In doing so, the real numbers $ \lambda ^ {(1)}, \lambda ^ {(2)}, \ldots, \lambda ^ {(\mu)} $ may be taken as both positive and linearly independent. Now, if $ F $ is a series containing only \textit{positive} powers of $ x, y, z, \ldots $ or containing at most finitely many negative terms, then it is absolutely convergent in a certain region as soon as (49) possesses a domain of convergence; we see that in this case (50) (and (49)) do always have a domain of absolute convergence. In particular, this is always the case if (49) is a series of ``number theoretical type'':

$$ \sum \frac {a_n} {n ^ s}. $$

Next, let (49) satisfy the algebraic differential equation with numerical coefficients 
\begin{align}
\Phi (f ^ {(i)} (s)) = 0.
\end{align}
We next substitute indefinites $u_i$ for the numbers $a_i$ in the expansion (49), and plug the resulting series $\varphi(s)$, which naturally is only of formal meaning, into $\Phi$.  Following a purely formal rearrangement, an expansion of the following form arises:

$$ \sum A_ {n_1, n_2, \ldots, n_ \mu} e ^ {- (n_1 \lambda ^ {(1)} + n_2 \lambda ^ {(2)} + \ldots + n_ \mu \lambda ^ {(\mu)}) s}, $$
where the exponentials are in increasing order of their exponents, and the integers $ n_1, n_2, \ldots, n_ \mu $ take on some positive and negative values. $ A_ {n_1, n_2, \ldots, n_ \mu} $ is a polynomial in the $ u_i $, the $ \lambda_i $ and the coefficients of $ \Phi $. However, not all products of powers of the $u_i$ can occur in $ A_ {n_1, n_2, \ldots, n_ \mu} $.  We only permit those that satisfy a certain weighted constraint. To derive this, we represent the general exponent $ \lambda_i $ of (49) linearly using $ \lambda ^ {(1)}, \lambda ^ {(2)}, \ldots, \lambda ^ {(\mu)} $ and integer coefficients:
$$ \lambda_i = \alpha_i ^ {(1)} \lambda ^ {(1)} + \alpha_i ^ {(2)} \lambda ^ {(2)} + \ldots + \alpha_i ^ {(\mu)} \lambda ^ {(\mu)}. $$
Since the exponents remain unchanged upon the differentiation of any term in (49), then every product of powers
$$ u_ {j_1} u_ {j_2} u_ {j_3} \ldots, $$
of the $ u_i $ in $ A_ {n_1, n_2, \ldots, n_ \mu} $, where $ j_1, j_2, j_3, \ldots $ are indices that are not necessarily mutually distinct, must satisfy the following system of constraints:
$$ \alpha_ {j_1} ^ {(\nu)} + \alpha_ {j_2} ^ {(\nu)} + \alpha_ {j_3} ^ {(\nu)} + \ldots = n_ \nu \hspace{30pt} (\nu = 1,2, \ldots, \mu), $$
due to the the linear independence of the $\lambda^{(\kappa)}$.
If $ c_1, c_2, \ldots, c_ \mu $ are arbitrary numbers, multiply each $ u_i $ by
$$ c_1 ^ {\alpha_i ^ {(1)}} c_2 ^ {\alpha_i ^ {(2)}} \ldots c_ \mu ^ {\alpha_i ^ {(\mu)}}. $$
Then each $ A_ {n_1, n_2, \ldots, n_ \mu} $ has been multiplied by $ c_1 ^ {n_1} c_2 ^ {n_2} \ldots c_ \mu ^ {n_ \mu} $. If all $ A $ vanish upon replacing the indeterminates $ u_i $ with the numbers $ a_i $, they also vanish if all $ u_i $ are replaced with $ a_ic_1 ^ {\alpha_i ^ {(1)}} c_2 ^ { \alpha_i ^ {(2)}} \ldots c_ \mu ^ {\alpha_i ^ {(\mu)}} $.  Any Dirichlet series that arises in this manner thus satisfies the differential equation (51) formally. However, the resulting series are
$$ F (c_1e ^ {- \lambda ^ {(1)} s}, c_2e ^ {\lambda ^ {(2)} s}, \ldots, c_ \mu e ^ {- \lambda ^ {(\mu )} s}), $$
and converge for sufficiently small $ c_1, c_2, \ldots, c_ \mu $ in the case where $ F $ consists of positive powers of $ x, y, z, \ldots $ and is convergent. We formulate this as

\hypertarget{Theorem 18} {}
\begin{theorem}
Let $ F (x, y, z, \ldots) $ be a power series iterated over positive integer powers of $ x, y, z, \ldots $ (with at most finitely many negative powers) and converging in a neighborhood of the origin.  Suppose that the Dirichlet series
\begin{align}
F (e ^ {- \lambda ^ {(1)} s}, e ^ {\lambda ^ {(2)} s}, e ^ {\lambda ^ {(3)} s}, \ldots),
\end{align} 
   where $ \lambda ^ {(1)}, \lambda ^ {(2)}, \lambda ^ {(3)}, \ldots $ are linearly independent and positive, satisfies an algebraic differential equation with numerical coefficients.  Then the function
\begin{align}
F (c_1e ^ {- \lambda ^ {(1)} s}, c_2e ^ {\lambda ^ {(2)} s}, c_3e ^ {- \lambda ^ {(3)} s}, \ldots),
\end{align}
satisfies the same equation for any choice of constants $ c_1, c_2, c_3, \ldots $, provided it can be defined for the chosen values of said constants. If $ F (x, y, z, \ldots) $ is a formal Laurent expansion, and (52) formally satisfies an algebraic differential equation with constant coefficients, then every series (53) also formally satisfies the same differential equation.
 \end{theorem}

From this theorem we can immediately conclude the following:

\hypertarget {Theorem 19} {}
\begin {theorem}
Let $ F (x, y, z, \ldots) $ be a power series iterated over positive integer powers of $ x, y, z, \ldots $ (or containing at most finitely many negative powers) that converges in a neighborhood of the origin.  Now suppose that for any system of positive linearly independent $ \lambda ^ {(1)}, \lambda ^ {(2)}, \ldots $, the Dirichlet series
\begin {align}
f (s) = F (e ^ {- \lambda ^ {(1)} s},
e ^ {- \lambda ^ {(2)} s}, \ldots)
\end {align}
satisfies an algebraic differential equation
\begin {align}
\Phi (f ^ {(i)} (s)) = 0.
\end {align}
Then $ F (x, y, z, \ldots) $ satisfies  an algebraic partial differential equation. 
If $ F (x, y, z, \ldots) $ is  Laurent expansion and (54) formally satisfies an algebraic differential equation, then $ F (x, y, z, \ldots) $ formally satisfies an algebraic partial differential equation.
\end {theorem}

\begin {proof}
We can assume $ \Phi (f ^ {(i)}) $ to be a polynomial in the $ f ^ {(i)} (s) $ with constant coefficients. Then $ F (c_1e ^ {- \lambda ^ {(1)} s}, c_2e ^ {- \lambda ^ {(2)} s}, \ldots) $ also satisfies the differential equation (55). Now every derivative with respect to $ s $ of $ F (c_1e ^ {- \lambda ^ {(1)} s}, c_2e ^ {- \lambda ^ {(2)} s}, \ldots) $ consists of partial derivatives of $ F (x, y, z, \ldots) $
with respect to its arguments, for which the values $ c_1e ^ {- \lambda ^ {(1)} s}, c_2e ^ {- \lambda ^ {(2)} s}, \ldots $ are plugged in ex post facto, along with the exponentials $ c_1e ^ {- \lambda ^ {(1)} s}, \ldots $ and the numbers $ \lambda ^ {(i)} $. If (55) is of order $ m $, we observe that the $ m $ -th derivative of
$ F (c_1e ^ {- \lambda ^ {(1)} s}, \ldots) $
is equal to:

$$ \sum [\frac {\partial ^ m F (x, y, z, \ldots)}
{\partial x ^ a \partial y ^ b \ldots}]
_ {\substack {x = c_1e ^ {- \lambda ^ {(1)} s} \\
y = c_2e ^ {- \lambda ^ {(2)} s}} \\
\ldots}
c_1 ^ a (- \lambda ^ {(1)}) ^ a
e ^ {- a \lambda ^ {(1)} s}
c_2 ^ b (- \lambda ^ {(2)}) ^ b
e ^ {- b \lambda ^ {(2)} s}
\ldots
+ \ldots
$$
 The $ m $ -th partial derivatives of $ F (x, y, z , \ldots) $ do not yet appear in the lower order derivatives of $ F (c_1e ^ {- \lambda ^ {(1)} s}, \ldots) $ with respect to $ s $. Next, by setting $ s = 0 $, (55) results in an algebraic $m$-th order partial differential equation for $ F (c_1, c_2, c_3, \ldots) $ as a function of
the constants $ c_1, c_2, c_3, \ldots $ which can once again be replaced by the variables $ x, y, z, \ldots $.
\end {proof}

We now return to the power series that are iterated over arbitrarily increasing powers of $ x $ and represent algebraic-transcendent functions. A result applicable to such series and corresponding to the first half of \hyperlink{Theorem 19}{Theorem 19}, can be directly transferred and verified from said theorem. However, we will propose a direct fashion in which the theorems corresponding to \hyperlink{Theorem 18}{Theorems 18} and \hyperlink{Theorem 19}{19} can be obtained in greater generality. Let $ \lambda_1, \ldots, \lambda_ \mu $ be arbitrary \textit{linearly independent} numbers, potentially complex. Let $ F (x, y, z, \ldots) $ be an arbitrary Laurent expansion in $ \mu $ variables $ x, y, \ldots $, and assume
$ F (x ^ {\lambda_1}, x ^ {\lambda_2}, \ldots) $ satisfies
an algebraic differential equation in $ x $. We moreover assume that
$ F (x ^ {\lambda_1}, x ^ {\lambda_2}, \ldots) $, expanded in powers of $ x $, has a sequence of exponents with a single accumulation point at infinity. This is to say that
$$
F (x ^ {\lambda_1}, x ^ {\lambda_2}, \ldots)
= f (x)
= \sum a_ix ^ {n_i}
\hspace{30pt}
\lim_ {i \to \infty} n_i = \infty
$$
\textit{Assume $ \lambda_i $ real at first and let $ f (x) $ (and by extension every derivative of $f(x)$) have a domain of absolute convergence.} By winding around the origin $m$ times, 
$ F (x ^ {\lambda_1}, x ^ {\lambda_2}, \ldots) $
becomes the integral
$ F (e ^ {2m \pi \lambda_1i} x ^ {\lambda_1},
e ^ {2m \pi \lambda_2i} x ^ {\lambda_2},
\ldots) $ of the 
 same differential equation
$ \Phi (f ^ {(i)} (x)) = 0 $. The coefficients of $ \Phi $ may be taken as rational in $ x $. If no integer linear combination of the $\lambda_i$ is rational, we can take the exponentials
$ e ^ {2m \pi \lambda_1i},
e ^ {2m \pi \lambda_2i},
\ldots $
to be arbitrarily close to the numbers $ c_1, c_2, \ldots $ of absolute value $1$ using a well-known theorem on Diophantine approximations. If we thus let $m$ become infinite in a suitable manner, it follows that
$ F (c_1 x ^ {\lambda_1},
c_2 x ^ {\lambda_2}, \ldots) $ also satisfies the differential equation $ \Phi (f ^ {(i)} (x)) = 0 $.
From this the same inferences as above result in $ F (x, y, \ldots) $ satisfying an algebraic partial differential equation. If the $\lambda_i$ satisfy a relation of the form
$ p_1 \lambda_1 + p_2 \lambda_2 + \ldots
+ p_ \mu \lambda_ \mu = p, $
where $ p_1 \ldots p_ \mu, p $ are integers and $ p \neq0 $, we can introduce arbitrary constants $c_1,c_2,...,c_{\mu - 1}$ into the integral
$ F (x ^ {\lambda_1}, \ldots,
x ^ {\lambda_ \mu}) $ in the same manner as above:
$ F (c_1 ^ {p_ \mu} x ^ {\lambda_1},
c_2 ^ {p_ \mu} x ^ {\lambda_2}, \ldots,
c_1 ^ {- p_1} c_2 ^ {- p_2} \ldots
c _ {\mu-1} ^ {- p _ {\mu-1}} x ^ {\lambda_ \mu}) $. 
Once again it follows by substituting $ x = cx '$ that 
$ F (\gamma_1 x ^ {\lambda_1}, \ldots,
\gamma_ \mu x ^ {\lambda_ \mu}) $ thus satisfies
an algebraic differential equation in $ x $ if $ \gamma_1, \ldots, \gamma_ \mu $
are arbitrary constants of absolute value $1$, - although not necessarily the same equation $ \Phi (f ^ {(i)}) = 0 $. From this we again conclude that $ F (x, y, z, \ldots) $ satisfies an algebraic partial differential equation. In both cases $ F (x, y, z, \ldots) $ need not have a convergence domain in the usual sense, in which case $ F (x, y, z, \ldots)$ \textit{formally} satisfies an algebraic partial differential equation.\\

We now drop the assumption of the absolute convergence of $ f (s) $, and assume that $ f (s) $ \textit{formally} satisfies an algebraic differential equation. However, we additionally assume that all but finitely many exponents $ n_i $ have the same angle from within the aperture $ y <\pi $. At the same time as $ f (s) $ we will consider the series
$$\varphi (x) = \sum u_i x ^ {n_i},$$
where $ u_i $ are indefinite. By plugging $ \varphi (x) $ into $ \Phi (\varphi ^ {(i)}) $, the series $\sum A_k(u_i)x^{m_k},$ emerges, where $ \lim_ {i \to \infty} m_i = \infty $ and $ A_k (u_i) $ only depends on a finite number of $ u_i $,
as with our assumption about the $ n_i $. $ A_k (a_i) = 0 $ for every $ k $.  Once again, we separate our problem into two cases:
\begin {enumerate}
\item No integer linear combination of the $ \lambda_i $ is rational, or there is a rational linear expression of the $ \lambda_i $ that is rational, in which case $ \Phi $ cannot explicitly depend on $ x $.
By replacing $ F (x ^ {\lambda_1}, \ldots,
x ^ {\lambda_ \mu}) $
with
$ F (c_1x ^ {\lambda_1}, \ldots, c_ \mu x ^ {\lambda_ \mu}) $,
 where $ c_1, \ldots, c_ \mu $ are arbitrary constants, each $ a_i $ becomes multiplied by a product of powers of the $ c_i $; at the same time the $ A_k (a_i) $ are simply multiplied by a product of powers of the $ c_i $, meaning they all remain equal to 0. We see that $ F (c_1 x ^ {\lambda_1}, \ldots,
c_ \mu x ^ {\lambda_ \mu}) $ thus also satisfies the differential equation $ \Phi (f ^ {(i)}) $. And from this we can in turn conclude that $ F (x, y, z, \ldots) $ satisfies an algebraic partial differential equation. 
\item However, if there is a linear relation
$ p_1 \lambda_1 + p_2 \lambda_2 + \ldots
+ p_ \mu \lambda_ \mu = p $,
where $ p_1 \ldots p_ \mu, p $
are integers and p is non-zero, and if $ \Phi $ explicitly contains the independent variable $ x $, we again construct the series
\begin{align}
F (c_1 ^ {p_ \mu} x ^ {\lambda_1},
c_2 ^ {p_ \mu} x ^ {\lambda_2}, \ldots,
c_1 ^ {- p_1} c_2 ^ {- p_2} \ldots
c _ {\mu-1} ^ {- p _ {\mu-1}} x ^ {\lambda_ \mu}),
\end{align}
where $c_1,c_2,...,c_{\mu-1}$ are to be understood as arbitrary non-zero constants.
If we plug this series into $ \Phi $, then the expressions $ A_k (a_i) $ are only multiplied by a product of powers of the $ c_i $, again remaining 0. Therefore, (56) also formally satisfies the differential equation $ \Phi (f ^ {(i)}) $, and from here, as above, we again conclude, in the case of $f(s)$ converging absolutely, that $ F (x, y, z, \ldots) $ formally satisfies an algebraic partial differential equation. If $ F (x, y, z, \ldots) $ also has a domain of convergence in the usual sense, then the representation of $F(x,y,z,\ldots)$ as a \textit{function} also satisfies the aforementioned differential equation.
\end {enumerate}

The results above on power series which are iterated over arbitrary powers of the independent variables are of interest insofar as such series expansions occur very often when studying the behavior of integrals of ordinary differential equations near a singular point. Specifically here, Poincaré, Picard and others\footnote{cf. E. Picard, Traité d'Analyse \textbf {3}, (2. éd.), Pp. 1–40} developed a method to reduce the task of finding series expansions for such integrals to the construction of functions of several variables $F(x,y,z,...)$ which satisfy certain partial differential equations that can be easily derived from what has been prescribed.  Additionally, we wish to obtain series expansions along the formula
$$ f (x) = F (c_1 x ^ {\lambda_1},
c_2 x ^ {\lambda_2}, \ldots)$$
from these.
However, these expansions rely on a great variety of different assumptions in the work of Poincaré and Picard, and the appearance of partial differential equations seems like a very odd gimmick. The results given above provide a certain \textit{reversal} of these methods in the case of algebraic differential equations - In the next section we will delve into the case of analytic differential equations, and in particular we will generalize \hyperlink{Theorem 12}{Theorem 12} to the case of complex exponents. – \\

We now return to the proof of \hyperlink{Theorem 19}{Theorem 19}, but now make the assumption that \textit{for sufficiently large positive $ \lambda ^ {(i)} $}, the function (54)
$$F (e ^ {- \lambda ^ {(1)} s},
e ^ {- \lambda ^ {(2)} s}, \ldots)$$
\textit{satisfies the differential equation (55), which is assumed independent of the $ \lambda ^ {(i)} $.} Similar to the previous case we again introduce arbitrary constants into the integral (54), but these may now be denoted by $ x, y, z, \ldots $. We thus get the function
$$ f (s) = F (x e ^ {- \lambda ^ {(1)} s},y e ^ {- \lambda ^ {(2)} s},\ldots) $$
which satisfies an algebraic differential equation $ \Phi (f ^ {(i) }) = 0 $ in $s$ which is independent from  $ x, y, z, \ldots, \lambda ^ {(1)}, \lambda ^ {(2)}, \ldots $. We now convert the differential equation into a partial differential equation in $ x, y, z, \ldots $, where $ s $ may remain indefinite for the time being. We thus get a partial differential equation in which the highest - $ m $-th - partial derivatives of $ F $ occur in the expression:
$$
\sum [\frac
{\partial ^ m F (x', y', \ldots)}
{\partial {{x'} ^ a} \partial{ {y'} ^ b} \ldots}]
_ {\substack {x '= x e ^ {- \lambda ^ {(1)} s} \\
y '= y e ^ {- \lambda ^ {(2)} s} \\
\ldots}}
x ^ a y ^ b \ldots
(- \lambda ^ {(1)}) ^ a
(- \lambda ^ {(2)}) ^ b
\ldots
e ^ {- a \lambda ^ {(1)} s
-b \lambda ^ {(2)} s- \ldots}
$$
By expressing $ x, y, \ldots $ as $ x '= xe ^ {- \lambda ^ {(1)} s},
y '= ye ^ {- \lambda ^ {(2)} s}, \ldots $,
we obtain a partial differential equation for $ F (x ', y', z ', \ldots) $, in which the highest derivatives of $ F $ occur in the expression:
$$
\sum (\frac
{\partial ^ m F (x ', y', \ldots)}
{\partial {x '}^ a \partial {y'} ^ b \ldots})
{x'}^ a {y'} ^ b \ldots
(- \lambda ^ {(1)}) ^ a
(- \lambda ^ {(2)}) ^ b
\ldots
$$
Since this differential equation holds for arbitrary values of $ \lambda ^ {(i)} $, it allows the $ m $ -th partial derivatives of $ F (x ', y', \ldots) $ to be expressed algebraically using the derivatives of lower order and $ x ', y', \ldots $ - \textit{$ F (x ', y', \ldots) $ thus satisfies an algebraic Mayer system.} 
This result is important because we shall now prove: \textit{if an analytic function $ F (x, y, \ldots) $ has the property that for any linearly independent system of real values $ \lambda_1, \lambda_2, \ldots $ occurring in any prescribed intervals, the function $ F (e ^ {- \lambda_1 s}, e ^ {- \lambda_2 s}, \ldots) $ of $ s $ is algebraic-transcendent, then $ F (e ^ {- \lambda_1 s}, e ^ {- \lambda_2 s}, \ldots) $ satisfies a fixed algebraic differential equation for all $ \lambda_i $.}
This does not follow directly from \hyperlink{Proposition 5}{Proposition 5}, since we want to assume $ \lambda_1, \lambda_2, \ldots $ as linearly independent, but the proof proceeds similarly as that of \hyperlink{Proposition 5}{Proposition 5}. We want to specifically consider the case of three variables and assume $ \lambda_1 $ equals $-1$, which is permitted in the theorems from \S 5. If $ \lambda_3 $ is set as any fixed irrational value $ \lambda_0 $, then $ F (e ^ s, e ^ {- \lambda_2 s}, e ^ {- \lambda_0 s}) $ satisfies  an algebraic differential equation for each irrational $ \lambda_2 $ from the prescribed interval that is linearly independent of $1,$ and $\lambda_0$. Hence there exist products of powers 
$$ \Pi_1, \Pi_2, \ldots, \Pi_k $$
of the the derivatives of $ F (e ^ s, e ^ {- \lambda_2 s}, e ^ {- \lambda_0 s}) $ which become linearly dependent functions of $s$ for infinitely many $ \lambda_3 $.
Therefore their Wronskian determinant vanishes identically in $ s, \lambda_2 $, and there is a relation
$$ c_1 (\lambda_2) \Pi_1
+ c_2 (\lambda_2) \Pi_2
+ \ldots + c_k (\lambda_2) \Pi_k = 0, $$
from which we finally get a differential equation for $ F (e ^ s, e ^ {- \lambda_2 s}, e ^ {- \lambda_0 s})$ through differentiation and the elimination of $ c_i (\lambda_2) $, as was justified in \S6, which $ F (e ^ s, e ^ {- \lambda_2 s}, e ^ {- \lambda_0 s}) $ satisfies for every value of $ \lambda_2 $. Since such a differential equation exists for every irrational value $ \lambda_0 $ of $ \lambda_3 $ from a given interval, a repetition of the same procedure leads us to an algebraic differential equation satisfied by $ F (e ^ s, e ^ {- \lambda_2 s}, e ^ {- \lambda_3 s}) $ for every value of $ \lambda_2, \lambda_3 $. And from this it follows, according to the theorems of \S 5, that $ F (e ^ {- \lambda_1 s}, e ^ {- \lambda_2 s}, e ^ {- \lambda_3 s}) $ satisfies a fixed algebraic differential equation for any system of values $ \lambda_1, \lambda_2, \lambda_3 $. - We also note that through successive differentiation and eliminations, which can be justified without any additional work using Lemma \ref{lemma:sec3} from \S3, we can say that the coefficients of the algebraic differential equation satisfied by $ F (e ^ {- \lambda_1 s }, e ^ {- \lambda_2 s}, e ^ {- \lambda_3 s}) $ are independent of $ \lambda_1, \lambda_2, \lambda_3 $ and thus become ordinary rational numbers. - We can now formulate the proposition that stands next to \hyperlink{Theorem 15}{Theorem 15} and \hyperlink{Proposition 5}{Proposition 5}:

\hypertarget {Theorem 20} {}
\begin {theorem}
If an analytic function $ F (x, y, z, \ldots) $ has the property that $ F (e ^ {- \lambda_1 s}, e ^ {- \lambda_2 s}, e ^ {- \lambda_3 s}, \ldots ) $ is algebraic-transcendent for any real $s$, or even just for any real linearly independent $ \lambda_1, \lambda_2, \lambda_3, \ldots $ located in prescribed intervals, then $ F (x, y , z, \ldots) $ itself an algebraic-transcendent function of $ x, y, z, \ldots $, thus satisfies a Mayer system of algebraic partial differential equations.
\end {theorem}

\newpage

\section{Power series not satisfying any analytic differential equations.}

The main goal of the developments within this section is the proof of the following, which will later be supplemented by \hyperlink{Theorem 22}{Theorem 22} with respect to questions of convergence.

\hypertarget{Theorem 21} {}
\begin {theorem}
Let $ F (x, y, y ', \ldots, y ^ {(n)}) $ be a power series iterated along positive powers of $ x, y-c, y'-c_1, \ldots, y ^ {(n )} - c_n $ and which converges in a certain neighborhood of the point $ x = 0, y = c, y '= c_1, \ldots, y ^ {(n)} = c_n $. In the differences $ y ^ {(i)} - c_i $, $ c_i $ are constant, some of which may be infinite. In this case, $ y ^ {(i)} - c_i $ is, as usual, to be understood as $ \frac1 {y ^ {(i)}} $. Suppose that the series
\begin {align}
y = \sum_ {i = 0} ^ \infty a_i x ^ {n_i}
\end {align}
formally satisfies the differential equation
\begin {align}
F (x, y, y ', \ldots, y ^ {(n)}) = 0.
\end {align}
Here the $ n_i $ are real or complex constants about which the following assumptions have to be made:
\begin{enumerate}
    \item Starting with a given $i$, their real parts are sorted in increasing order and converge to $ + \infty $.
    \item If not all $ n_i $ are real, let $ n_k $ be a non-real exponent whose real part $ r $ is as small as possible. If the real exponents $ n_i $ that are smaller than $ r $ are non-negative integers, or if there are no real exponents less than $ r $, we must assume that $ r $ has no non-negative integer value, and that there is only one exponent $ n_i $ whose real part is equal to $ r $.
\end{enumerate}
Let the numbers $ c, c_1, \ldots $ be the ``values'' of $ y $ and the derivatives of $ y $ for $ x = 0 $. \\

If one of the numbers $ c_i $ (and thus each subsequent one) is $ \infty $, then the formally constructed $ i $-th derivative of $ y $ must start with a power of $ x $ whose exponent has a negative real part . Then, according to our assumptions about the $n_i$, $ \frac {1} {y ^ {(i)}} $ can be expanded as a series iterated over powers of $ x $, where all exponents have positive real parts. These expansions are then to be plugged into $ F $ in place of $ \frac {1} {y ^ {(i)}} $. \\

Under these assumptions, all $ n_i $ can be expressed linearly with integer coefficients by a finite number of them.
\end {theorem}

\begin {proof}
We denote the quantities $ x, y-c, \ldots, y ^ {i} -c_i, \ldots $ (or $ \frac {1} {y ^ {(i)}} $) by $ z_1, z_2, \ldots, z_ {i + 2}, \ldots $ and can then write (58) in the form:
\begin {align}
\Phi (z_1, z_2, \ldots) = 0,
\end {align}
where $ \Phi $ is an ordinary power series in $ z_1, z_2, \ldots $, which vanishes at the origin and converges in a neighborhood of the origin. We denote the first partial derivatives of $ \Phi $ with respect to $ z_1, z_2, \ldots $ by $ \Phi_1 (z_i), \Phi_2 (z_i), \ldots $ and claim \textit{that equation (59) can, if necessary, be replaced such that none of the equations hold formally:}
\begin {align}
\Phi_1 (z_i) = 0, \hspace {20pt}
\Phi_2 (z_i) = 0, \ldots
\end {align}
 The real difficulty of the following investigations lies in verifying the correctness of this claim. \\

We must first recall some facts from the theory of the divisibility of power series in several variables in the neighborhood of a point\footnote {See, for example, Weierstrass. Some theorems relating to the theory of the analytical functions of several variables, Collected Works, 2, p. 135-188}. - We assume that all power series to be considered converge in a neighborhood of the origin. - A power series $ \psi_1 (z_i) $ is said to be \textit{divisible} by another $ \psi_2 (z_i) $, if an equation exists:
$$ \psi_1 (z_i) = \psi_2 (z_i) m (z_i), $$
where $ m (z_i) $ is a power series. If each of the power series $ \psi_1 (z), \psi_2 (z) $ is divisible by the other, they are called \textit{equivalent}. Then $ m (z_i) $ does not vanish at the origin. A power series $ \psi (z_i) $ is called \textit{irreducible} if it is only divisible by those power series that are equivalent to it, i.e. if it cannot be written in the form:
$$ \psi (z_i) = \psi_1 (z_i) \psi_2 (z_i) $$
where both power series $ \psi_1 (z_i) $ and $ \psi_2 (z_i) $ vanish at the origin.
If one does not regard equivalent irreducible power series as possibly being distinct, then each power series that vanishes at the origin can be represented uniquely as a product of irreducible power series.
The proof of this last theorem relies on Weierstrass's \textit{preparation theorem}. The preparation theorem states: If for a power series $ \psi (z_1, z_2, z_3, \ldots) $ that vanishes at the origin, a power of $ z_1 $ occurs by itself, i.e. $ \psi (z_1,0,0, \ldots) $ does not vanish identically in $z_1$, then $ \psi (z_1, z_2, \ldots) $ is equivalent to a power series that is a \textit{polynomial} with respect to $ z_1 $. That is to say that the following applies:
$$ \psi (z_1, z_2, z_3, \ldots)
=
(z_1 ^ m + A_1 (z_2, z_3, \ldots) z_1 ^ {m-1}
+ \ldots
+ A_m (z_2, z_3, \ldots)) E (z_1, z_2, \ldots), $$
where $ E (z_1, z_2, z_3, \ldots) $ does not vanish at the origin and $ A_1 (z_2, z_3, \ldots), \ldots,$
$A_m (z_2, z_3, \ldots) $ are power series in $ z_2, z_3, \ldots $. - If, however, there is no isolated power of $z_1$, or of $ z_2 $, etc. in $ \psi (z_1, z_2, \ldots) $, a linear transformation of the variables $ z_i $ can be used so that for each of the new variables of $ \psi $ a power that is free of the other variables occurs, in which case the preparation theorem becomes applicable again. \\

If the preparation theorem were always applicable without an initial linear transformation, the proof of the claims made above would not yield any difficulties. In that case, we could eliminate the variable $ z_ {n + 2} $ from the equations $ \Phi (z_i) = 0 $ and $ \Phi_i (z_i) = 0 $ and thereby get an analytic differential equation, which $y$ formally satisfies and which would be of \textit{lower} order. Subsequently, a final application of this procedure would lead us to our goal. But since the preparation theorem is not always applicable, it seems that our current knowledge on analytic functions of several variables does not always allow the direct elimination of a \textit{specified} variable from two analytic equations\footnote {Let us mention the following problem: Two ordinary power series $ f (x, y, z), g (x, y, z) $ converge in a neighborhood of the origin and vanish at the origin. Does there always exists a power series $ h (x, y) $ that converges in a neighborhood of the origin, and vanishes there for all (and only) those pairs of sufficiently small $ x, y $, such that both power series $ f, g $ vanish for sufficiently small $z$? For example: One considers the curve constructed by the intersection of two surfaces that have an algebraic character at the origin. Does its projection onto any plane passing through the origin always have an algebraic character at the origin?}. In our case, however, it is possible to reach the goal by means of a detour. - \\

It is possible that a power series in $ \mu $ variables $ z_i $ is equivalent to a power series depending on fewer than $ \mu $ variables after a linear homogeneous invertible transformation of the $ z_i $. The \textit{rank} of $\psi(z_i)$ will denote the smallest number $ \varrho $ with the property that $ \psi (z_i) $ is equivalent to a power series in $ \varrho $ variables after a linear homogeneous invertible transformation of the $ z_i $.\\

If $ \psi (z_i) $ is of rank $ \varrho $, then every divisor of $ \psi (z_i) $ is also of rank $ \varrho $. Since if we transform the variables $ z_i $ into new variables $ \zeta_i $, every divisor of $ \psi (z_i) $ becomes a divisor of the transformed power series; but if the transformed power series is equivalent to one that only depends on $ \varrho $ variables $ \zeta_i $, then the same applies to each of its divisors, according to the theorem about the uniqueness of the decomposition of a power series into irreducible factors. \\

We now return to our question and choose a differential equation $ \Phi (z_i) $ formally satisfied by $y$, with $ \Phi (z_i) $ being of as small a rank $ \varrho $ as possible. If $ \Phi (z_i) = \Phi ^ {(1)} (z_i) \Phi ^ {(2)} (z_i) $, where $ \Phi ^ {(1)} $ and $ \Phi ^ {( 2)} $ are power series, then $ y $ formally satisfies one of the equations
$$ \Phi ^ {(1)} (z_i) = 0, \hspace {30pt}
\Phi ^ {(2)} (z_i) = 0. $$
Therefore we can take $ \Phi (z_i) $ to be an irreducible power series of rank $ \varrho $. There thus exists a linear homogeneous invertible transformation
\begin{align}
z_i = \sum_ {ik} \zeta_k, \hspace {30pt}
\zeta_i = \sum s_ {ik} z_k,
\end{align}
by which $ \Phi (z_i) $ is mapped to a power series
$ \overline {\Phi} (\zeta_i) $
to which the representation
$$ \Phi (z_i)
= \overline {\Phi} (\zeta_i)
= \Psi (\zeta_i) E (\zeta_i)$$
applies, where $ \Psi (\zeta_i) $ only depends on $ \varrho $ of the variables $ \zeta_i $, say $ \zeta_1, \ldots, \zeta_\varrho $, and $E$ does not vanish at the origin. If $ \frac {1} {E (\zeta_i)} = e (z_i) $, by replacing $ \Phi (z_i) $ with $ \Phi (z_i) e (z_i) $ if necessary, we can convert $ \Phi (z_i) $ directly into a power series $ \Psi (\zeta_i) $ which only depends on $ \varrho $ variables $ \zeta_1, \ldots, \zeta_ \varrho $, \textit{directly} using $(61)$. \textit{The differential equation $ \Phi (z_i) = 0 $ then has the desired property.} Since assuming this, $ y $ also formally satisfies an equation $ \Phi_1 (z_i) = \frac {\partial \Phi (z_i)} {\partial z_i} = 0 $ for instance. From the identity:
$$ \Phi_1 (z_i) = \frac {\partial \Phi (z_i)} {\partial z_i}
= \sum s_ {k1} \frac {\partial \Psi (\zeta_i} {\partial \zeta_k} = \Psi ^ * (\zeta_i) $$
it follows that $ \Phi_1 (z_i) $ becomes a power series $ \Psi ^ * (\zeta_i) $ in $ \zeta_1, \ldots, \zeta_ \varrho $ using (61). We now apply another linear homogeneous substitution to $ \Psi (\zeta_i) $ and $ \Psi ^ * (\zeta_i) $
\begin {align}
\zeta_i = \sum \tau_{ik} u_k,
\hspace {30pt}
u_i = \sum t_{ik} \zeta_k,
\hspace {30pt}
(i = 1, \ldots \varrho; k = 1, \ldots, \varrho),
\end {align}
through which we achieve that Weierstrass's preparation theorem becomes applicable to both power series. Namely, we can achieve that both $ \Psi $ and $ \Psi ^ * $, after the transformation, contain isolated powers of the same variable, say $u_1$. Since if the sums of the terms of the lowest dimension in $ \Psi (\zeta_i) $ and $ \Psi ^ * (\zeta_i) $ are $ a (\zeta_i) $ and $ b (\zeta_i) $, we only need (62) to be set up in such a way that after the transformation (62), all products of powers of $ u_i $ that are permissible according to dimension actually occur in the homogeneous forms $ a $ and $ b $ . We can therefore write
\begin{align*}
	\Psi(\zeta_i)=X(u_i)&=E(u_i)\Pi(u_i)=E(u_i)(u_1^\mu+A^{(1)}(u_i)u_1^{\mu-1}+\ldots+A^{(\mu)}(u_i))\\
\Psi^*&=X^*(u_i)=E^*(u_i)\Pi^*(u_i)=\\
&=E^*(u_i)(u_1^\nu+A^{*(1)}(u_i)u_1^{\nu-1}+\ldots+A^{*(\nu)}(u_i)).
\end{align*}
Here $ E (u_i), E ^ * (u_i) $ are power series that do not vanish at the origin, $A ^ {(i)}$ and $ A ^ {* (i)} $ however are power series only in $ u_2, \ldots, u_ \varrho $. If we now form the resultant $ P (u_2, \ldots, u_ \varrho) $ from $ \Pi $ and $ \Pi ^ * $ with respect to $ u_1 $, the following identity holds:
$$ P (u_2, \ldots, u_ \varrho)
= B (u_i) \Pi (u_i)
+ B ^ * (u_i) \Pi ^ * (u_i) $$
or
\begin{align}
P (u_2, \ldots, u_ \varrho)
= C (u_i) X (u_i)
+ C ^ * (u_i) X ^ * (u_i).
\end{align}

If we now go back to the $ \zeta_i $ and $ z_i $, $ P (u_2, \ldots, u_ \varrho) $ yields a power series $ R (z_i) $ whose rank is at most $ \varrho-1 $, and for which (63) gives an identical relation of the form
$$ R (z_i) =
D (z_i) \Phi (z_i)
+ D_1 (z_i) \Phi_1 (z_i). $$
From this it would follow, however, that $ y $ also satisfies the equation $ R (z_i) = 0 $ formally, while $ R (z_i) $ is of the rank $ \varrho-1 $. Thus the above claim has been proven. \\

We can therefore assume that none of the equations (60) are formally satisfied. 
Therefore, if one plugs the series expansion for $ y $ into one of the power series $ \Phi_i (z) $ in (60), and sorts the result according to the real parts of the exponents of the individual terms, which is possible according to our assumptions, not all terms vanish identically, and the expansion will start with terms in which the real parts of the exponents are as small as possible, say $ \nu_i $. 
Let these terms be:
\begin {align}
c_ {i, \kappa} x ^ {\lambda_ {i, \kappa}}.
\end {align}
From our premises it thus follows: If one truncates the series for $ y $ somewhere sufficiently far along, the initial terms of the expansion of $ \Phi_i (z) $ do not change. The exponents $ \lambda_ {i, \kappa} $ are thus homogeneous integer linear combinations of certain $ n_i $. We now split the series for $ y $ into two parts
$$ y =
A_t (x) + R_t (x)
= \sum_ {i = 0} ^ {t-1} a_ix ^ {n_i} + \sum_ {i = t} ^ \infty a_ix ^ {n_i} $$
and take $ t $ to be so large that starting with this $t$, when $ A_t(x) $ is substituted for $ y $, the initial terms of the $ \Phi_i $ are the sums of the expressions (64). We now denote $ A_t (x), A_t '(x), \ldots, A_t ^ {(n)} (x) $ as $ u_2, u_3, \ldots, u_ {n + 2}, $ and $ R_t ( x), R_t '(x), \ldots, R_t ^ {(n)} (x) $
as $ v_2, v_3, \ldots, v_ {n + 2} $, where the indices are chosen in correspondance with those of $ z_i $. The $ u_i $ and $ v_i $ do still depend on $ t $. \\

We want to expand $ \Phi (z_i) $ in terms of powers of $ v_i $. For this purpose we consider that as long as the real part of the initial term of $ y ^ {(i)} $ is not negative, $ z_ {i + 2} = y ^ {(i)} - c_i = (u_ {i +2} -c_i) + v_ {i + 2} $, so that we simply have to use Taylor's theorem directly to expand into powers of these $ v_ {i + 2} $. If, on the other hand, the initial term $ b_ix ^ {\mu_i} $ of $ y ^ {(i)} $ has an exponent with a negative real part, then
\begin {align}
z_ {i + 2} = \frac {1} {u_ {i + 2} + v_ {i + 2}}
= \frac1 {b_i} x ^ {- \mu_i}
\frac {1} {\frac {u_ {i + 2}} {b_i x ^ {\mu_i}} + \frac {v_ {i + 2}} {b_i x ^ {\mu_i}}}
= \frac1 {b_i} x ^ {- \mu_i}
\sum_ {k = 0} ^ {\infty}
(\frac {u_ {i + 2}} {b_i x ^ {\mu_i}} - 1+ \frac {v_ {i + 2}} {b_i x ^ {\mu_i}}) ^ k.
\end {align}
Here $ \frac {u_ {i + 2}} {b_i x ^ {\mu_i}} - 1 $ and $ \frac {v_ {i + 2}} {b_i x ^ {\mu_i}} $ have many terms with different exponents, and the real parts of the exponents of all terms of $ \frac {v_ {i + 2}} {b_i x ^ {\mu_i}} $ are, as soon as $ t $ is taken sufficiently large, greater than the exponents of the terms of $ \frac {u_ {i + 2}} {b_i x ^ {\mu_i}} $. We finally note that the expansions of the derivatives of $ z_ {i + 2} $ in $ v_ {i + 2} $, formally constructed from (65),  for $ v_ {i + 2} = 0 $, have exponents whose real parts are not negative. The formula (65) shows that one can use Taylor's theorem when expanding $ \Phi (z_i) $ in powers of $ v_i $. For $ i \geq 0 $ we denote generally $ u_ {i + 2} -c_i $ or $ \frac1 {u_ {i + 2}} $ as $ w_ {i + 2} $ (thus $ w_ {i + 2} = (z_ {i + 2}) _ {v_ {i + 2} = 0} $), then the expansion
\begin {align}
\Phi (z) = \Phi (x, w) + \sum v_i (\frac {\partial \Phi (x, z)} {\partial v_i}) _ {z_s = w_s} + \sum v_i v_k (\frac {\partial ^ 2 \Phi (x, z)} {\partial v_i \partial v_k}) _ {z_s = w_s} + \ldots
\end {align}
holds.  Here all differentiations are to be carried out in such a way that one first differentiates in $ z_i $ and then differentiates $ z_i $ in $ v_i $ (either from $ z_i = u_i + v_i $, or from $ z_i = \frac {1} {u_i + v_i} $). Next we consider, however, that the exponents of the individual terms in the expansions of the derivatives (in $v_i$) of $ z_i $ have a non-negative real part. Likewise, the partial derivatives of $ \Phi (z_i) $ with respect to the $ z_i $ only contain terms whose exponents have non-negative real parts, and this is even the case when we replace the $ z_i $ with the $ w_i $ . From all this it follows that the real parts of the exponents of the individual terms in the sums

$$
\sum v_i v_k (\frac {\partial ^ 2 \Phi (x, z)} {\partial v_i \partial v_k}) _ {z_s = w_s}, \hspace {30pt}
\sum v_i v_k v_l (\frac {\partial ^ 2 \Phi (x, z)} {\partial v_i \partial v_k \partial v_l}) _ {z_s = w_s}, \ldots
$$
\textit{are no smaller than} $ 2 \sigma_t $, where $ \sigma_t $ is the real part of the exponent of $ a_t ^ {n_t} $.\\

We have applied Taylor's theorem above, even the meaning of the expansion (66) is of course quite different from what is commonplace. The $ u_i $ and $ v_i $ are not sufficiently small \textit{numbers}, but rather power series in $ x $. Equation (66) says that if you replace the $z_i$ and the $v_i$ with their series expansions in $ x $ on the right and left-hand sides, and rearrange them purely formally, exactly the same terms appear on the right and left-hand sides. However, this rearrangement is always possible without taking limits, since according to our assumptions about the $ u_i $, only a \textit{finite} number of terms can occur that have the same exponent. \\

We could also interpret equations (65), (66) as\textit{ congruences} modulo sufficiently large positive powers $ x ^ \varrho $ of $ x $, in the sense that for all powers of $x$, we are to ignore those for whose exponents the real parts are larger than $\varrho$. Then on both sides of equations (65) and (66) only finitely many terms are considered for any $\varrho$, and equations (65), (66) say that the corresponding congruences hold for each $\varrho $. - And since the expansions resulting in equation (66) always have coefficients that are power series in $ x $ in which the exponents of the individual terms have non-negative real parts, it is clear why we are allowed to use Taylor's theorem (and the rule on differentiating function with respect to a function). \\

We now draw our attention to the first sum on the right hand side of (66)

$$ \sum v_i (\frac {\partial \Phi (x, z)} {\partial v_i}) _ {z_s = w_s}. $$
If for any $ i> 1 $ we have that $ z_i = u_i + v_i $, then the coefficient of $ v_i $ equals
$$ \Phi_i (x, w_2, \ldots, w_ {n + 2}). $$
But if $ z_i = \frac {1} {u_i + v_i} $, then the coefficient of $ v_i $ equals
$$ \Phi_i (x, w_2, \ldots, w_ {n + 2}) \cdot \frac {-1} {u_ {i + 2} ^ 2}. $$
We now wish to find the terms in (66) whose exponents have the smallest real part. For $ i $ such that $ z_i = u_i + v_i $, the initial terms of $ v_i (\frac {\partial \Phi} {\partial v_i}) _ {z_s = w_s} $ are equal to
\begin {align}
(i-2)! \binom {n_i} {i-2} a_t x ^ {n_t-i + 2} c_ {i, \kappa} x ^ {\lambda_ {i, \kappa}}
\hspace {30pt} (\kappa = 1,2, \ldots)
\end {align}
for sufficiently large $t$.
If the real part of $ n_t $ is $ \sigma_t $, then the real part of the exponents of the terms (67) is $ \sigma_t + \nu_t-i + 2 $. But if $ z_i = \frac {1} {u_i + v_i} $, then 
\begin {align}
- (i-2)! \binom {n_i} {i-2} a_t x ^ {n_t-i + 2} c_ {i, \kappa} x ^ {\lambda_ {i, \kappa}}
\frac {1} {b_i ^ 2} x ^ {- 2 \mu_i}
\hspace {30pt} (\kappa = 1,2, \ldots),
\end {align}
where $ b_i x ^ {\mu_i} $ is the initial term of $ u_i $. If the real part of $ \mu_i $ equals $ \varrho_i $, then the real part of the exponents of (68) is equal to
$$ \sigma_t + \nu_i-2 \varrho_i-i + 2. $$
Here the numbers $ \nu_i, \varrho_i $ are independent of $ t $. We now take the sum of the terms (67) and (68), bringing $ a_t x ^ {n_t} $ in front of the brackets, then subsequently sorting the terms inside the brackets in order of their exponents. We then obtain
$$ a_t x ^ {n_t} \sum \varphi_s (n_t) x ^ {k_s}, $$
where $ k_s $ are all distinct complex numbers and $ \varphi_s (n_t) $ are polynomials in $ n_t $ with fixed numerical coefficients that are independent of $ t $. Evidently, not all $ \varphi_s (n_t) $ vanish identically in $ n_t $, since in (67) and (68) $ i $ takes on many different values, and therefore $ n_t $ appears only once in the $ n $-th power, namely $ \binom {n_t} {n} $. If $ k $ is a $k_s$ for which the real part is as small as possible, and the coefficient $ \varphi (n_t) $ of $ x ^ k $ is non-zero, then $ \varphi (n_t) $ can only vanish for a finitely many $ n_t $. From a certain $ t $ onward, the term
$$ a_t \varphi (n_t) x ^ {n_t + k} $$
either cancels with a term from the second, third, $\ldots$ sum on the right-hand side of (66), or with a term of $ \Phi (x, w) $. The real parts of the exponents of the terms in the second, third, $ \ldots $ sum of the right-hand side of (66) are no less than $ 2 \sigma_t $, that is from a certain $ t $ onward they are certainly greater than the real part of $ \sigma_t + k $. Hence, the expansion of $ F (x, w_2, \ldots, w_ {n + 2}) $ must result in at least one term whose exponent is equal to $ n_t + k $. The exponents of the terms $ \Phi (x, w) $ are made up of integer linear combinations of $ 1, n_0, n_1, \ldots, n_ {t-1} $. Likewise, $ k $ is made up of $ 1 $ and finitely many $ n_i $ linearly using integer coefficients.

\end{proof}

In \hyperlink{Theorem 21}{Theorem 21} we formulated the property that, for very extensive classes of series iterated over arbitrary powers of the argument, they do not \textit{formally} satisfy any analytic partial differential equation. We now want to extend this property to \textit{functions} that can be represented by such series. \\

To achieve this, we first of all require of the exponents $ n_i = \sigma_i + \sqrt {-1} \tau_i $ that $ \lim_ {i \to \infty} \sigma_i = + \infty $, secondly that $ | \frac {\tau_i} {\sigma_i} | $ remains bounded from a certain $ i $ on. This means, geometrically speaking, that all $ n_i $ from a certain $ i $ onwards lie within an angle from the aperture $ \gamma <\pi $, containing positive part of the real axis, while leaving the imaginary axis (except for the origin) completely untouched. \\

In order to formulate the assumption regarding the convergence of the series (57), we introduce the concept of \textit{absolute convergence in the sharp sense}. We call a series $ \sum a_i x ^ {n_i} $ with complex exponents $ n_i = \sigma_i + \sqrt {-1} \tau_i $, which satisfy the above conditions, absolutely convergent in the sharp sense for $ | x | = \nu $, if the series $ \sum | a_i | \nu ^ {\sigma_i} e ^ {2 \pi | \tau_i |} $ converges. Such a series converges absolutely on the whole circle $|x| = \nu $, as well as in every non-zero point within this circle, and absolutely in the sharp sense on every concentric circle of smaller radius. If $ \sum a_ix ^ {n_i} $ converges absolutely in the sharp sense for $ | x | = \nu $, then the formally constructed derivative $ \sum n_i a_ix ^ {n_t-1} $ of this series converges absolutely in the sharp sense for $ | x | = \nu- \epsilon $, where $ \epsilon $ is an arbitrarily small fixed positive quantity. Since in order to prove this, from the convergence of $ \sum \alpha_i \nu ^ {\sigma_i} e ^ {2 \pi | \tau_i |} $, where $ | a_i | = \alpha_i $ is set, the convergence of the series $ \sum \alpha_i | \sigma_i + \sqrt {-1} \tau_i | (\nu- \epsilon) ^ {\sigma_t-1} e ^ {2 \pi | \tau_i |} $ must be deduced. But this follows from
$$
\lim_ {i \to \infty} | \sigma_i + \sqrt {-1} \tau_i |
\frac {(\nu- \epsilon) ^ {\sigma_i-1}} {\nu ^ {\sigma_i}}
= | \frac {\sigma_i + \sqrt {-1} \tau_i} {\sigma_i} | \sigma_i
\frac {(\nu- \epsilon) ^ {\sigma_i-1}} {\nu ^ {\sigma_i}} = 0.
$$
Since $ \frac {\sigma_i + \sqrt {-1} \tau_i} {\sigma_i} $ is absolutely bounded from a certain $ i $ onward and $ \lim_ {i \to \infty} \sigma_i (\frac {\nu- \epsilon} {\nu}) ^ {\sigma_i} = 0. $ It follows from this that every formally constructed derivative of $ \sum a_i x ^ {n_i} $ is absolutely convergent in the sharp sense for $ | x | = \nu - \epsilon $. And from the uniformity of the convergence, which results from the proof of the uniqueness theorem below, it follows that the formally constructed derivatives represent the derivatives of the function represented by the original series. \\

If $ \sum a_i x ^ {n_i} $ converges absolutely in the sharp sense, and if all $ \sigma_i $ are positive, then, as we shall immediately prove in detail, the majorant $ \sum \alpha_i \nu ^ {\sigma_i} e ^ {2 \pi | \tau_i |} $ also becomes arbitrarily small with decreasing $ \nu $. From this it follows that the result of substituting (57) in the left-hand side of (58) converges absolutely in the sharp sense on a certain circle around the origin, and there represents the result of substituting the function represented by (57) in (58). Thus it remains only to prove the uniqueness: If a series $ \sum a_i x ^ {n_i} $ converges absolutely in the sharp sense on a circle around the origin , and if the function represented by it equals 0, then all coefficients $ a_i$ vanish . \\

We can obviously assume that all $ n_i = \sigma_i + \sqrt {-1} \tau_i $ are mutually distinct and that $ \sigma_i $ are ordered in non-decreasing order. We can also assume that the first $ k (k> 0) \sigma_i $ vanish, since this otherwise can be achieved by multiplying by $ x ^ {- \sigma_0} $. Thus let $ \sigma_0 = \ldots = \sigma_k = 0 $, $ \sigma_ {k + 1}> 0 $. We claim first that every positive $ \epsilon $ can be assigned such a $ \delta $ that $ | \sum_ {k + 1} ^ \infty a_i x ^ {n_i} | <\epsilon $ as soon as $ | x | <\delta $. In general, let $ | a_i | = \alpha_i $. From the absolute convergence in the sharp sense of $ \sum a_ix ^ {n_i} $ it follows that the Dirichlet series
\begin {align}
\sum_ {i = k + 1} ^ \infty \alpha_i e ^ {2 \pi | \tau_i |} e ^ {- \sigma_i s}
\end {align}
converges absolutely and therefore uniformly in a certain $ s $-half-plane. From this it follows that one can specify $ \delta_1 $ such that the absolute value of (69) becomes smaller than $ \epsilon $ as soon as the real part of $ s $ becomes greater than $ \delta_1 $. From this it follows, however, that the absolute value of $ \sum_ {i = k + 1} ^ \infty a_i x ^ {n_i} $ becomes smaller than $ \epsilon $ as soon as $ | x | <\delta = e ^ {- \delta_1} $. \\

As a result, if all the coefficients $ a_0, \ldots, a_k $ were non-zero, the sum $ a_0 x ^ {\sqrt {-1} \tau_0} + \ldots + a_k x ^ {\sqrt {-1} \tau_k} $ would converge to $ 0 $ with $ x $. So it suffices to show that there is such a sequence of real values of $ s $ that increase to infinity such that
$$ f (s) = a_0x ^ {\sqrt {-1} \tau_0} + \ldots
+ a_k x ^ {\sqrt {-1} \tau_k} $$
converges to a value other than $ 0 $. The entire function $ f (s) $ cannot vanish identically, since
otherwise it would also have to vanish for purely imaginary $ s = \sqrt {-1} t $, and one term in $ f (s) $ would dominate for sufficiently large $ t $ , since $ \tau_0, \ldots, \tau_k $ are mutually distinct. For real $ s = s_0 $, assume $ f (s_0) \neq 0 $. Let $ m $ be an arbitrarily large positive integer. Then, as is well-known, there is a non-zero positive integer $ s_m $, which can be found with the help of the famous Dirichlet approximation, such that the numbers $ s_m \tau_0, \ldots, s_m \tau_k $ differ from certain integer multiples of $ 2 \pi $ by less than $ \frac {1} {m ^ 2} $. And consequently the value of $ f (s_0 + ms_m) $ will converge to $ f (s_0) \neq 0 $ with increasing $ m $.

We can summarize this result as

\hypertarget {Theorem 22} {}
\begin {theorem}
If an absolutely convergent series in the sharp sense
$$ y (x) = \sum a_i x ^ {n_i} \hspace {30pt} (n_i = \sigma_i + \sqrt {-1} \tau_i) $$
has the property that
\begin {enumerate}
\item $ \lim \sigma_i = + \infty $;
\item $ | \frac {\tau_i} {\sigma_i} | $ remains bounded from a certain $ i $ onward;
\item if $ \sigma_i $ is the smallest $ \sigma $ corresponding to a non-zero $ \tau $, and if all smaller $ \sigma $ are non-negative integers, then $ \sigma_i $ is not a non-negative integer and there is no second exponent $ n $ with the same $ \sigma_i $;
\end {enumerate}
and if furthermore the function $ y (x) $ satisfies a differential equation, the left-hand side of which is analytic in $ x, y-c, y'-c_1, \ldots, y ^ {(n)} - c_n, \frac {1} {y ^ { (n + 1)}}, \ldots, \frac {1} {y ^ {(n + m)}} $, where the real parts of the exponents of the initial terms of $ x, y-c, y'-c_1, \ldots, y ^ {(n)} - c_n, \frac {1} {y ^ {(n + 1)}}, \ldots, \frac {1} {y ^ {(n + m)}} $ are non-negative and $ c_1, \ldots, c_n $ are finite constants, then the exponents $ n_i $ have a finite integer basis.
\end {theorem}

At the same time, this theorem contains a generalization of \hyperlink{Theorem 12}{Theorem 12} to the case where the exponents are complex\footnote {In fact, every algebraic differential equation whose left-hand side is a polynomial in $ x $, the unknown function and its derivatives, can be written in the form (58), possibly by division with a product of powers of the unknown function and its derivatives.}. The substitution $ x = e ^ {- s} $ then leads to an analogous theorem about Dirichlet series $ \sum a_i e ^ {- \lambda_i s} $, where $ \lambda_i $ may be complex, and an analogous  requirement of absolute convergence in the sharp sense must be made. \\

In formulating the theorem at the beginning of this section, we made assumptions about the form of the differential equation (58), which could possibly be formulated a little more generally. To that end, the theorem that among the exponents $ n_i $ there are only finitely many linearly independent ones probably holds even if $ F $ is a power series iterated over any integer (whether positive or negative) powers of its arguments converging in a certain neighborhood of the origin. However, this case no longer permits use of the methods used in this section. The reason why the validity of \hyperlink{Theorem 21}{Theorems 21} and \hyperlink{Theorem 22}{22} also appears plausible in this case is easy to overlook. If there were an infinite number of linearly independent exponents among the exponents $ n_i $, then we would be able to (at least formally) introduce an infinite number of arbitrary constants into the integral by repeatedly winding around the origin. This conclusion can be carried out completely in many cases. For example, with their help one can prove that the function $ \sum_ {i = 1} ^ \infty x ^ {\text {log} (p_i)} $, where $ p_i $ is the $ i $-th prime number, satisfies no differential equation whose left-hand side is an unambiguous, generally analytic function of the derivatives in a neighborhood of the origin. And analogous power series in several variables can be constructed. For example, the analogous fact can be proven for the series
$$ \sum x ^ n y ^ {\text {log} (n)}, $$
which arises from the series $ \sum x ^ n n ^ {- s} $ by a simple substitution. And with that a new - third - proof of the property that the series $ \sum x ^ n n ^ {- s} $ does not satisfy any algebraic partial differential equation has been given. I will elaborate on these remarks elsewhere and in doing so I will delve into their close connection with a known conclusion, first used by Mr. H. Bohr in the theory of Riemann's $ \zeta $ function.

\begin {center}
(Received February 3, 1919.)
\end {center}

\end{document}